\documentclass[preprint,11pt]{elsarticle2}
\usepackage{fullpage}

\usepackage{amssymb,amsmath,mathtools,amsthm}
\usepackage{hyperref,stackengine}

\newtheorem{theorem}{Theorem}[section]
\newtheorem{lemma}[theorem]{Lemma}

\newtheorem{definition}[theorem]{Definition}

\DeclareMathOperator{\du}{d\!}

\DeclareMathOperator{\bu}{{\mathbf{u}}}
\DeclareMathOperator{\bq}{{\mathbf{q}}}
\DeclareMathOperator{\blf}{\mathbf{f}}
\DeclareMathOperator{\bw}{\mathbf{w}}
\DeclareMathOperator{\bv}{\mathbf{v}}

\DeclareMathOperator{\bn}{\mathbf{n}}

\DeclareMathOperator{\bg}{\mathbf{g}}

\DeclareMathOperator{\dive}{\mathrm{div}}

\DeclareMathOperator{\ws}{ \mathrel{ \ensurestackMath{ \stackon[1pt]{\rightharpoonup}{\scriptstyle\ast}}}}

\hypersetup{
	colorlinks,
	linkcolor={red!50!black},
	citecolor={blue!50!black},
	urlcolor={blue!80!black}
}

\numberwithin{equation}{section}

\begin{document}

\begin{frontmatter}



\title{On a nonlocal two-phase flow with convective heat transfer}


\author{{\v S}{\'a}rka Ne{\v c}asov{\' a}}
\ead{matus@math.cas.cz}

\author{John Sebastian H. Simon }
\ead{simon@math.cas.cz}

\affiliation[]{organization={Institute of Mathematics, Czech Academy of Sciences},
            city={Prague},
            country={Czech Republic}}

\begin{abstract}
We study a system describing the dynamics of a two-phase flow of incompressible viscous fluids influenced by the convective heat transfer of Caginalp-type. The separation of the fluids is expressed by the order parameter which is of diffuse interface and is known as the Cahn-Hilliard model.
We shall consider a nonlocal version of the Cahn-Hilliard model which replaces the gradient term in the free energy functional into a spatial convolution operator acting on the order parameter and incorporate with it a potential that is assumed to satisfy an arbitrary polynomial growth. 
The order parameter is influenced by the fluid velocity by means of convection, the temperature affects the interface via a modification of the Landau-Ginzburg free energy.
The fluid is governed by the Navier--Stokes equations which is affected by the order parameter and the temperature by virtue of the capillarity between the two fluids.
The temperature on the other hand satisfies a parabolic equation that considers latent heat due to phase transition and is influenced by the fluid via convection.
The goal of this paper is to prove the global existence of weak solutions and show that, for an appropriate choice of sequence of convolutional kernels, the solutions of the nonlocal system converges to its local version. 
\end{abstract}



\begin{keyword}
nonlocal Cahn--Hilliard equations \sep Boussinesq equations \sep well-posedness \sep nonlocal-to-local convergence

\MSC[2020] 45K05 \sep 76D03 \sep 76T06 \sep 35B40

\end{keyword}

\end{frontmatter}


\section{Introduction}

The study of multi-phase flows has been of interest among experts due to their complexities and robust properties. One of the well-known models is the so-called Cahn-Hilliard equation which models a binary immiscible fluid that assumes a diffuse interface between the two fluids \cite{cahn1958,cahn1959}. Since then several modifications have been introduced to take into account the physical and biological implications of the model, see eg. \cite{rocca2023,rocca2021} for application to tumor growth, \cite{abels2009a,abels2022} for binary fluids with unmatched densities, \cite{kalousek2021,mitra2023} for binary magnetic fluids, \cite{ohta1986} for separation of diblock copolymers, and \cite{bertozzi2007} for image reconstruction. 

Aside from modifying the original system or coupling it with other physical models to describe more physical (or biological) phenomena, the Cahn-Hilliard equations have been recasted to consider long-range particle interactions on the interface of the two fluids. This version of the Cahn-Hilliard equations, known to many as the nonlocal Cahn-Hilliard equations, was introduced by G. Giacomin and J. Lebowitz \cite{Giacomin1997,Giacomin1998}. Such model has been the subject of many expositions, among them are the following works \cite{abels2015,bates2005,burkovska2021,colli2012, DAVOLI2021a,Davoli2021b,frigeri2016}. These modifications serve as motivations for this article. In particular, we will consider an incompressible two-phase flow undergoing phase separation due to diffusion and is influenced by temperature by means of a Caginalp-type equation. The dynamics of the binary fluids described by the Cahn-Hilliard equations will make use of the aforementioned nonlocal system.

Let $\Omega\subset \mathbb{R}^d$ be a fixed domain -- with sufficiently smooth boundary $\partial\Omega$ -- which a two-phase fluid occupies, and $T>0$ be a fixed final time. 
We denote by $({\bu},p):\Omega\times(0,T) \to\mathbb{R}^d\times\mathbb{R}$ the velocity-pressure pairing of the fluid, $\theta:\Omega\times(0,T)\to\mathbb{R}$ the relative temperature around the critical temperature $\theta_c = 0$, and $\varphi:\Omega\times(0,T)\to\mathbb{R}$
the (relative) concentration of the binary fluid. 

The concentration $\varphi$ is described by the Cahn--Hilliard equation with a transport term influenced by the average velocity ${\bu}$, which takes values $\varphi\in [-1,1]$, where $\varphi = 1$ or $\varphi = -1$ corresponds to the two different phases of the fluids.
To be specific, the dynamics of $\varphi$ are governed by 
\begin{align*}
	D_t\varphi = m\Delta \mu,
\end{align*}
where $D_t = \partial_t + {\bu}\cdot\nabla$ is the material derivative due to the fluid velocity, $\mu$ is the chemical potential, and $m$ corresponds to the mobility of the interface. Such model -- if one, for the meantime, ignores the contribution of the fluid velocity -- is based on the so-called H-model which considers a diffuse interface between two interacting fluids \cite{cahn1958,gurtin1996}.

The average velocity and pressure of the fluids is governed by the Navier--Stokes equations and is influenced by the order parameter $\varphi$ and the relative temperature $\theta$ by a Korteweg-type force proportional to $(\mu-\ell_c\theta)\nabla\varphi$ where $\ell_c$ is a parameter concerning latent heat \cite[Appendix]{jasnow1996}. A Boussinesq approximation which incorporates a constant gravitational force with the linearized equation of state expressed as $\ell(\varphi,\theta){\bg} = (\alpha_0 + \alpha_1\varphi + \alpha_2\theta){\bg}$ is also assumed to affect the fluid. To be specific, the fluid is governed by the equation
\begin{align*}
	D_t{\bu} - \dive(\nu(\varphi)2\mathrm{D}{\bu}) + \nabla p = \mathcal{K}(\mu-\ell_c\theta)\nabla\varphi + \ell(\varphi,\theta){\bg} + {\bq}.
\end{align*}
Here $\mathcal{K}$ is a capillarity constant, $\nu(\varphi)>0$ denotes the fluid viscosity, ${\bq}$ denotes an external force and the operator $\mathrm{D}$ is defined as $\mathrm{D}{\bu} = (\nabla{\bu} + \nabla{\bu}^\top)/2$. Ignoring for the meantime the contribution of the temperature $\theta$, the combination of the two equations previously introduced is the well-known Navier--Stokes/Cahn--Hilliard system which has been considered in numerous expositions, see for example the papers \cite{abels2009a, kalousek2021, abels2009b,garcke2012,boyer2002,boyer2001,gal2010, giorgini2019} and the references therein.

The equation concerning the temperature is based on Caginalp's interpretation of free boundary problems derived from phase trasition \cite{caginalp1986}. By introducing the quantity $H(\theta,\varphi):= \theta - \ell_h\varphi$ we consider the dynamics governed by the transport equation
\begin{align*}
	D_tH(\theta,\varphi) -\kappa \Delta\theta = {\bg}\cdot{\bu} + z,
\end{align*}
where $\ell_h$ is a constant related to the latent heat, ${\bg}$ is the gravitational force, and $z$ is an external heat source. The Laplacian above comes from the assumption that the heat flux follows Fourier's thermal conduction law.

Before we go any further, let us talk about the chemical potential $\mu$. Here, we note that the usual Landau--Ginzburg free energy functional is written as
\begin{align*}
	\tilde{\mathcal{E}}(\varphi,\theta) := \int_\Omega \left(\frac{\xi}{2}|\nabla\varphi|^2 + \eta F(\varphi) + \ell_c\theta\varphi \right)\du x,
\end{align*}
where the chemical potential $\mu$ is obtained by the first variation of $\mathcal{E}(\cdot,\theta)$, i.e., $\mu = -\xi\Delta \varphi + \eta F'(\varphi) + \ell_c\theta$. Such functional is used in \cite{peralta2021,Peralta2022} to model a nonisothermal diffuse interface two-phase flow where the author showed the existence of solutions and its application to optimal control problems.

Recent developments, however, saw the applicability of nonlocal energy functionals, for example in tumor growth models \cite{rocca2021} and in exhibiting sharp interface without having to taking asymptotic limit of the diffusive constant \cite{burkovska2021}. In this paper, we shall modify the temperate Landau free energy functional into a nonlocal one by changing the gradient into a nonlocal spatial operator. To be specific, we shall use the non-local Landau--Ginzburg free energy functional 
\begin{align*}
	\mathcal{E}(\varphi,\theta) :=&  \frac{1}{4}\int_\Omega\int_\Omega J(x-y)(\varphi(x)-\varphi(y))^2 \du x\du y + \int_\Omega\left( \eta F(\varphi(x) ) + \ell_c\theta(x)\varphi(x) \right)\du x,
\end{align*}
where $J:\mathbb{R}^2\to\mathbb{R}$ is a sufficiently smooth function that satisfies $J(x) = J(-x)$. The first variation of $\mathcal{E}(\cdot,\theta)$ will then give us the chemical potential $\mu = a\varphi - J\ast \varphi + \eta F'(\varphi)+ \ell_c \theta $, where 
\begin{align}
	a(x):= \int_\Omega J(x-y)\du y\ \text{ and }\ (J\ast\varphi)(x) = \int_\Omega J(x-y)\varphi(y)\du y.\label{defn:aundconv}
\end{align}
Such energy functional has been proven to take into account microscopic interactions that describe phase segregation \cite{Giacomin1997,Giacomin1998}. We also mention \cite{DAVOLI2021a,Davoli2021b} where the authors proved the convergence of the nonlocal system to the local one which also serve as a justification of the contention that $\tilde{\mathcal{E}}$ is the macroscopic limit of the energy functional $\mathcal{E}$.

Combining all these, we get the non-local Cahn-Hilliard-Boussinesq equations
\begin{subequations} 
\begin{align}
  \partial_t\varphi + {\bu}\cdot\nabla\varphi = m\Delta \mu,\quad  \mu = a\varphi - J\ast \varphi + \eta F'(\varphi)+ \ell_c \theta, \label{orderundpotential}
\end{align}   \vspace{-.2in}
\begin{align}
    \begin{aligned}
  \partial_t{\bu} + ({\bu}\cdot\nabla){\bu} {-}  \dive(\nu(\varphi)2\mathrm{D}{\bu}) + \nabla p = \mathcal{K}(\mu-\ell_c\theta)\nabla\varphi + \ell(\varphi,\theta){\bg} + {\bq},
  \end{aligned}
\end{align} \vspace{-.2in}
\begin{align}
 \partial_t\theta - \ell_h\partial_t\varphi + {\bu}\cdot\nabla(\theta - \ell_h\varphi) -\kappa \Delta\theta = {\bg}\cdot{\bu} + z,
\end{align}
\label{system:nlCHOB}
\end{subequations}
\!all of which are satisfied in the space-time domain $Q:= \Omega\times I$, where $I=(0,T)$, with the divergence-free assumption $\dive{\bu} = 0$ in $Q$, the boundary boundary conditions $\frac{\partial\mu}{\partial{\bn}} = 0$,  $ {\bu} = 0$, and $ \frac{\partial\theta}{\partial{\bn}}  = 0$ on $\Gamma:= \partial\Omega\times I$, and initial conditions $\varphi(0) = \varphi_0$, ${\bu}(0) = {\bu}_0$, and $\theta(0) = \theta_0$ in $\Omega$. 

The purpose of this article is to show the existence of weak solutions through spectral Galerkin method and to provide proof that, under appropriate conditions and choice of the convolutional kernel, solutions of the nonlocal system converge to its local version. We shall apply the techniques used in \cite{melchionna2019} for the Cahn-Hilliard system with periodic boundary conditions and sigular free energy density, by \cite{DAVOLI2021a} for the Cahn-Hilliard equations with singular potentials and $W^{1,1}$ kernels, and \cite{abels2022} for a Cahn--Hilliard--Navier--Stokes equations with unmatched densities.

 {\it The novelty in this article are as follows:} 

(i) compared to the works of P. Colli et al. \cite{colli2012} and S. Frigeri et al. \cite{frigeri2016} we considered a nonisothermal system, where the influence of the temperature is based on G. Caginalp \cite{caginalp1986}; 
(ii) aside from modifying the local free energy functional into the nonlocal version in \cite{peralta2021,Peralta2022}, we considered a non-constant viscosity dependent on the order parameter, a generalized potential of polynomial growth, and considered not only the two-dimensional case but also delved into the system in three dimensions.

The subsequent parts of this paper is as follows: the next section is devoted to the mathematical setting such as the functional spaces and the necessary embeddings. Section \ref{section:existence} deals with the existence of weak solutions for \eqref{system:nlCHOB} where we use the spectral Galerkin method as a discretization which gets to converge to the original continuous system. The final section establishes convergence of the nonlocal system to its local version by using the framework introduced in \cite{bourgain2001,bourgain2002} and the $\Gamma$-convergence in \cite{ponce2003,ponce2004}.

\section{Preliminaries}

Let us introduce the function spaces upon which the analysis will anchor on.
Given a Banach space $X$ with its dual $X^*$, the duality pairing is simply denoted as $\langle x^*,x\rangle_X$ for $x^*\in X^*$ and $x\in X$.
We denote by $L^p(\Omega)$ the space of Lebesgue integrable function of order $p\ge 1$ and the usual Sobolev-Slobodeckij spaces are denoted as $W^{s,p}(\Omega)$ with $H^s(\Omega) = W^{s,2}(\Omega)$. 
For simplicity, we use $\|\cdot\|$ and $(\cdot,\cdot)_\Omega$ to denote the norm and inner product in $L^2(\Omega)$ or $L^2(\Omega)^d$, for $d=2,3$, $\|\cdot\|_{L^p}$ for the norm in $L^p(\Omega)$ or $L^p(\Omega)^d$, and $\|\cdot\|_{W^{s,p}}$ the norm in $W^{s,p}(\Omega)$ or $W^{s,p}(\Omega)^d$.
To take into account the incompressibility of the fluid we utilize the following solenoidal spaces $V_\sigma = \{{\bu}\in H^1_0(\Omega)^d: \dive{\bu} = 0 \text{ in } \Omega \}$, and $H_\sigma = \{{\bu}\in L^2(\Omega)^d : \dive{\bu} = 0\text{ in }\Omega,\, {\bu}\cdot{\bn} = 0\text{ on }\partial\Omega \}$.

Let us define $V_s:= D(B^{s/2})$, where $B = -\Delta$ with homogeneous Neumann boundary conditions.
We hence have $$V_2 = D(B) = \left\{ \varphi\in H^2(\Omega) : \frac{\partial\varphi}{\partial{\bn}} =0 \text{ on }\partial \Omega \right\}.$$ 
We also use the notation  $H:= V_0 = L^2(\Omega)$ and $V:= V_1 = H^1(\Omega)$.
We note that $B:D(B)\subset H\to H$ is an unbounded linear operator in $H$, and $B^{-1}:H\to H$ is a self-adjoint compact operator on $H$, which is also maximal monotone. We obtain a sequence of eigenvalues $\nu_j$ with $0< \nu_1 \le \nu_2 \le \ldots \le \nu_j \to \infty$, with the associated eigenfunctions $\psi_j\in D(B)$ such that $B\psi_j = \nu_j \psi_j$ for all $j$. The set $\{\psi_j\}$ is an orthonormal eigenbasis for $H$ while orthogonal in $V$ and $D(B)$.

Let us introduce the Stokes operator $A:D(A)\subset H_\sigma\to H_\sigma$ defined as $A = -P\Delta$, where $P:L^2(\Omega)^d\to H_\sigma$ is the Leray projector.  We note that $D(A) = H^2(\Omega)^d\cap V_\sigma$ whenever $\Omega$ is of either of class $\mathcal{C}^2$ or a convex polygonal domain.
A sequence of eigenvalues $\lambda_j$ with $0\le \lambda_1 \le \lambda_2 \le \ldots \le \lambda_j \to \infty$ can be obtained with the associated eigenfunctions $v_j\in D(A)$ so that $Av_j = \lambda_jv_j$. The set $\{v_j\}$ is an orthonormal eigenbasis for $H_\sigma$ which is orthogonal in $D(A)$.

Let ${\bu}\in V_\sigma$ and $\varphi\in H$, the operator $\mathcal{A}:V_\sigma\times H\to V_\sigma^*$ is defined as $$\langle \mathcal{A}({\bu},\varphi),{\bv} \rangle_{V_\sigma} = 2(\nu(\varphi)\mathrm{D}{\bu},\mathrm{D}{\bv})_{\Omega}.$$ 
The trilinear form $b: [H^1(\Omega)^d]^3 \to \mathbb{R}  $ coming out of the weak formulation of the Navier--Stokes equations is defined as $b({\bu};{\bv},{\bw}) = (({\bu}\cdot\nabla){\bv},{\bw})_{\Omega}$. Such trilinear form is continuous and known to satisfy $b({\bu};{\bv},{\bw}) = - b({\bu};{\bw},{\bv}) $ for all ${\bu},{\bv},{\bw}\in V_\sigma$.
From $b$, we also define $\mathcal{B}:V_\sigma\times V_\sigma \to V_\sigma^*$ as $\langle \mathcal{B}({\bu},{\bv}),{\bw} \rangle_{V_\sigma} = b({\bu};{\bv},{\bw})$ for ${\bu},{\bv},{\bw}\in V_\sigma$. 
From which we shall use the notation $\mathcal{B}({\bu},{\bu}) := \mathcal{B}({\bu})$ for simplicity.

To take into account the transport terms involving the scalar functions, we introduce the operators $\mathcal{C}_1:V_\sigma\times V\to V^*$ and $\mathcal{C}_2:V\times V\to V_\sigma^*$ respectively defined as 
\begin{align*}
    &\langle \mathcal{C}_1({\bv},\varphi),\psi \rangle_V = ({\bv}\cdot\nabla \varphi,\psi)_\Omega,\\
    &\langle\mathcal{C}_2(\varphi,\psi),{\bv} \rangle_{V_\sigma} = ({\bv}\cdot\nabla \varphi,\psi)_\Omega = (\psi\nabla\varphi, {\bv})_\Omega.
\end{align*}

The following inequalities are well-known but we reiterate them here for completeness.
\begin{lemma}
    Let $\Omega\subset\mathbb{R}^n$ be an open and bounded domain that is at least of Lipschitzian class. 
    \begin{enumerate}
        \item[I.] Rellich--Kondrachov: If $1\le mp < n$ by setting $p^* = np/(n-mp)$ then the embedding $W^{m,p}(\Omega)\hookrightarrow L^{p^*}(\Omega)$ is continuous, while the embedding $W^{m,p}(\Omega)\hookrightarrow L^{q}(\Omega)$ is compact whenever $1\le q < p^*$. If $mp\ge n$ then we have the compact embedding $W^{m,p}(\Omega)\hookrightarrow L^{q}(\Omega)$ whenever $1\le q < \infty$.
        
        \item[II.] Gagliardo--Nirenberg: Let $1\le p_1,p_2,p_3\le +\infty$, $j,m\in \mathbb{N}$ with $j<m$  and $\theta\in [0,1]$ satisfy the relation
\begin{align*}
    \frac{1}{p_1} = \theta\left( \frac{1}{p_3} - \frac{m}{n} \right) + \frac{1-\theta}{p_2},\quad \frac{j}{m}\le \theta\le 1.
\end{align*}
Then if $u\in L^{p_2}(\Omega)\cap W^{m,p_3}(\Omega)$ we get, for some $c>0$ independent of $u$,
\begin{align*}
    \|u \|_{L^{p_1}} \le c\|u\|_{W^{m,p_3}}^\theta\|u\|_{L^{p_2}}^{1-\theta}.
\end{align*}

        \item[III.] Interpolation in $L^p$: Let $u\in L^{p_1}(\Omega)$, $p_2<p_1$ and $p_\theta = p_1p_2/[(1-\theta)p_1 + \theta p_2]$. There exists $c>0$ independent of $u$ such that
\begin{align*}
    \|u\|_{L^{p_\theta}}\le c\|u \|_{L^{p_1}}^\theta\|u\|_{L^{p_2}}^{1-\theta}.
\end{align*}
    \end{enumerate}
\end{lemma}

By virtue of Gagliardo--Nirenberg inequality the following properties hold true for any ${\bu},{\bv},{\bw}\in V_\sigma$:
\begin{align}
	&|b({\bu};{\bv},{\bw})| \le c\|{\bu}\|_{H_{\sigma}}^{1/2} \|{\bu}\|_{V_{\sigma}}^{1/2} \|{\bv}\|_{V_{\sigma}} \|{\bw}\|_{V_{\sigma}}, &&\text{for }d=3\label{trilinholder3d}\\
	&|b({\bu};{\bv},{\bw})| \le c\|{\bu}\|_{H_{\sigma}}^{1/2} \|{\bu}\|_{V_{\sigma}}^{1/2} \|{\bv}\|_{V_{\sigma}} |{\bw}\|_{H_{\sigma}}^{1/2} \|{\bw}\|_{V_{\sigma}}^{1/2}, &&\text{for }d=2.\label{trilinholder2d}
\end{align}

 Functions which are defined on the interval mapped to a real Banach space $X$ shall also be utilized. We denote by $C(I;X)$ the functions which are continuous from $I$ to $X$ upon which we use the norm $\|u\|_{C(I;X)} := \displaystyle\sup_{t\in\overline{I}}\|u(t)\|_X$. The Bochner spaces $L^p(I;X)$ with the norm 
\begin{align*}
\|u\|_{L^p(I;X)} = \left\{
\begin{aligned}
     &\left( \int_I \|u(t)\|_X\du t \right)^{1/p} &&\text{if }p<\infty,\\
     &\mathrm{ess}\sup_{t\in I}\|u(t)\|_X &&\text{if }p=\infty,
\end{aligned}\right.
\end{align*}
shall also be utilized, as well as the space $W^{m,p}(I;X):= \{u\in L^p(I;X): \partial^{j}_t u\in L^p(I;X), 0\le j\le m \}$.

We recall Aubin--Lions--Simon embedding. Let $X,Y,Z$ be Banach spaces such that the embeddings $X\hookrightarrow Y$ and $Y\hookrightarrow Z$ are compact and continuous, respectively. Given $1\le p,q\le \infty$, the space $W^{p,q}(X,Z):= \{ u\in L^p(I;X): \partial_t u\in L^q(I;Z)  \}$ is compactly embedded to $L^p(I;Y)$ if $p<+\infty$. On the other hand, if $p=+\infty$ and $q > 1$ the embedding $W^{p,q}(X,Z) \hookrightarrow C(\overline{I};Y)$ is compact.

\section{Existence of weak solution}\label{section:existence}

In this section we provide the existence of weak solutions to system \eqref{system:nlCHOB}.
By weak solution we use the following notion:

\begin{definition}
	Suppose that $\varphi_0\in H$ with $F(\varphi_0)\in L^1(\Omega)$, ${\bu}_0\in H_\sigma$, and $\theta_0 \in H$. We say that $[\varphi,{\bu},\theta]$ is a weak solution to \eqref{system:nlCHOB} (with the boundary condition incorporated) on $[0,T]$ corresponding to the initial conditions if 
	\begin{itemize}
		\item[$\bullet$] $[\varphi,{\bu},\theta]$ including $\mu$ satisfies
			\begin{align*}
				&\left\{
				\begin{aligned}
					&\varphi\in L^\infty(I;H)\cap L^2(I;V),\\
					&\mu \in L^2(I;V),\\
					&\partial_t\varphi\in L^{4/d}(I;V^*),
				\end{aligned}\right.\\
				&\left\{
				\begin{aligned}
					&{\bu}\in L^\infty(I;H_\sigma)\cap L^2(I;V_\sigma),\\
					&\partial_t{\bu}\in L^{4/d}(I;V_\sigma^*),
				\end{aligned}\right.\\
				&\left\{
				\begin{aligned}
					&\theta\in L^\infty(I;H)\cap L^2(I;V),\\
					&\partial_t\theta\in L^{4/d}(I;V^*),
				\end{aligned}\right.
			\end{align*}
		\item[$\bullet$] letting $\rho(x,\varphi):= a(x)\varphi + F'(\varphi)$ the following system of variational problems hold:
			\begin{subequations}
			    \begin{align}
                    \begin{aligned}
					&\langle \partial_t\varphi(t),\psi \rangle_V +  ({\bu}(t)\cdot\nabla\varphi(t),\psi)_\Omega + (m(\varphi)\nabla\rho(t),\nabla\psi)_\Omega  \\ & = (m(\varphi)\nabla(J\ast\varphi(t)),\nabla\psi)_\Omega - \ell_c(m(\varphi)\nabla\theta(t),\nabla\psi)_\Omega,
                    \end{aligned}\label{weak:phi}
				\end{align}
				\begin{align}
                    \begin{aligned}
					&\langle \partial_t{\bu}(t),{\bv} \rangle_{V_\sigma} + (({\bu}(t)\cdot\nabla){\bu}(t),{\bv} )_\Omega + 2(\nu(\varphi)\mathrm{D}{\bu}(t),\mathrm{D}{\bv}) \\& = \mathcal{K}(\bv\cdot\nabla\varphi(t), \mu(t) - \ell_c\theta(t) )_\Omega + (\ell(\varphi(t),\theta(t)){\bg},{\bv} )_\Omega + \langle {\bq},{\bv} \rangle_{V_\sigma}
                    \end{aligned}\label{weak:u}
				\end{align}
				\begin{align}
                    \begin{aligned}
					& \langle \partial_t\theta(t),\vartheta  \rangle_V -  \ell_h\langle \partial_t\varphi(t),\vartheta \rangle_V +  + \kappa(\nabla\theta(t),\nabla\vartheta)\\ & = ({\bu}(t)\cdot\nabla(\ell_h\varphi(t)-\theta(t)),\vartheta)_\Omega + ({\bg}\cdot{\bu}(t),\vartheta)_\Omega + \langle z,\vartheta  \rangle_V
                    \end{aligned}\label{weak:theta}
				\end{align} 			
                \end{subequations}
		for all $\psi\in V$, ${\bv}\in V_\sigma$, $\vartheta\in V$ and almost every $t\in (0,T)$
		\item[$\bullet$] the initial conditions hold in weak sense, i.e.
		\begin{align}
			&({\varphi}(t),{\psi})_{\Omega} \to ({\varphi}_0,{\psi})_{\Omega} \text{ as }t\to 0\quad \forall {\psi}\in V,\label{initcon:phi}\\
			&({\bu}(t),{\bv})_{\Omega} \to ({\bu}_0,{\bv})_{\Omega} \text{ as }t\to 0\quad \forall {\bv}\in V_\sigma,\label{initcon:u}\\
			&({\theta}(t),{\vartheta})_{\Omega} \to ({\theta}_0,{\vartheta})_{\Omega} \text{ as }t\to 0\quad \forall {\vartheta}\in V.\label{initcon:theta}
		\end{align}
	\end{itemize} 
\label{definition:weak}
\end{definition}

Note that $({\bu}(t)\cdot\nabla\varphi(t),\psi)_\Omega = -({\bu}(t)\cdot\nabla\psi,\varphi(t))_\Omega$. This implies that by taking $\psi = 1$, we get $\langle \partial_t\varphi(t),1 \rangle_V =0$. Whence, we get a total mass conservation for the order parameter.

To be able to provide such existence, we have to impose the following assumptions on the kernel $J$, the viscosity $\nu$, the potential $F$, and the external forces  ${\bq}$ and $z$:
\begin{itemize}
	\item[(A1)] $J\in W^{1,1}(\mathbb{R}^d)$, \quad $J(x) = J(-x)$, and $a(x)\ge 0$ a.e. $x\in\Omega$.
	\item[(A2)] We assume that the mobility is constant upon which we assume $m(\varphi) = 1$, while the viscosity function $\nu$ is locally Lipschitz on $\mathbb{R}$ and that there exists $\underline{\nu},\overline{\nu}>0$ such that $\underline{\nu}\le \nu(s) \le \overline{\nu}$ for all $s\in\mathbb{R}$.
	\item[(A3)] $F\in C^{2,1}_{loc}(\mathbb{R})$ and that there exists $c_0>0$ such that $F''(s) + a(x) \ge c_0$, for all $s\in\mathbb{R}$ and a.e. $x\in \Omega$.
	\item[(A4)] There exists $c_1>\frac{1}{2}\|J\|_{L^1(\mathbb{R}^d)}$ and $c_2\in\mathbb{R}$ such that $F(s) \ge c_1s^2 - c_2,$ for all $s\in\mathbb{R}$.
	\item[(A5)] There exists $c_3>0$, $c_4\ge 0$, and $p\in (1,2]$ such that $|F'(s)|^p \le c_3|F(s)| + c_4$ for all $s\in\mathbb{R}$.
	\item[(A6)] $\bq\in L^2(I;V_\sigma^*)$, and $z\in L^2(I;V^*)$.
\end{itemize}

We mention that (A1) is a standard assumption for the nonlocal Cahn-Hilliard equations which properly serves our purpose of establishing existence of weak solutions and is satisfied by the sequence of kernels for establising nonlocal-to-local convergence, nevertheless we mention \cite{bates2005} for a slightly stricter assumption and \cite[Remark 1]{frigeri2016} for a weaker assumption. 

Note that (A3) implies that $F$ can be written as $F(s) = G(s) - \frac{a^*}{2}s^2$, where $a^* = \|a\|_{L^\infty(\Omega)}$ and $G\in C^{2,1}(\mathbb{R})$ is a strictly convex function. While (A5) accounts for an arbitrary polynomial growth for $F$. 

We also mention Korn's inequality which will be useful in the subsequent analyses
\begin{lemma}
There exists a constant~$c_{\rm K} > 0$ such that the inequalities, $\|D(\bv)\| \le \|\bv\|_{V_\sigma} \le c_{\rm K} \|D(\bv)\|$, hold for any $\bv \in V$.
	\label{lemma:korn}
\end{lemma}

Lastly, let us introduce the total energy for the system which is given by:
\begin{align}\hspace{-.2in}
\begin{aligned}
2\mathbb{E}(\varphi,{\bu},\theta):=&\, \frac{1}{\ell_c}\left\{\frac{1}{2}\int_\Omega\int_\Omega J(x-y)(\varphi(x) - \varphi(y))^2 \du y\du x + 2\int_\Omega F(\varphi(x))\du x \right\}\\ & +
 \frac{1}{\mathcal{K}\ell_c}\int_\Omega |{\bu}(x)|^2 \du x + \frac{1}{\ell_h}\int_\Omega |\theta(x)|^2\du x,
\end{aligned}\label{totalenergynl}
\end{align}
and the dissipation functional denoted and written as:
\begin{align}
    \mathbb{D}(\varphi,\mu,{\bu},\theta): = \frac{1}{\ell_c}\|\nabla\mu \|^2 + \frac{2}{\mathcal{K}\ell_c}\|\sqrt{\nu(\varphi)}\mathrm{D}{\bu} \|^2 + \frac{\kappa}{\ell_h}\|\nabla\theta \|^2
\end{align}
We also use the following notation for the nonlocal energy
\begin{align*}
    \mathbb{E}_{nl}(\varphi) := \frac{1}{2}\int_\Omega\int_\Omega J(x-y)(\varphi(x) - \varphi(y))^2 \du y\du x.
\end{align*}

The proof of existence of weak solutions that we shall provide will be based to that of \cite[Theorem 1]{colli2012}. We provide a comprehensive proof for completeness as well as to simplify the computations we go through for the asymptotic analysis in the next section. We also point out that instead of beginning with a higher regularity for the initial data of the order parameter, as done in the aforementioned reference, we go directly into using Yosida approximation for the initial data. The promised existence of weak solution is summarized and proven below. 
\begin{theorem}
	Suppose that the assumptions on the initial data is as in Definition \ref{definition:weak}. Additionally, assume that (A1)-(A6) hold. Then, for any $T>0$ there exists a weak solution $[\varphi,{\bu},\theta]$ to the system \eqref{system:nlCHOB} on $[0,T]$. Furthermore, the following energy inequality holds for almost every $t>0$
	\begin{align}
	\begin{aligned}
		&\mathbb{E}(\varphi(t),{\bu}(t),\theta(t)) + \int_0^t \mathbb{D}(\varphi(s),\mu(s),{\bu}(s),\theta(s)) \du s\\ 
        &\le \mathbb{E}(\varphi_0,{\bu}_0,\theta_0)
		+ \int_0^t \big\{ \langle{\bq}(s),{\bu}(s)\rangle_{V_\sigma} + (\ell(\varphi(s),\theta(s)){\bg},{\bu}(s))_\Omega \\
        &\ \ + \langle{z}(s),{\theta}(s)\rangle_{V} + ({\bg}\cdot{\bu}(s),\theta(s))_\Omega \big\} \du s  .
	\end{aligned}\label{energyineq}
	\end{align}
 \label{th:existence}
\end{theorem}
\begin{proof}
The existence of solutions is established using spectral Galerkin method. We project the system \eqref{weak:phi}-\eqref{weak:theta} onto the finite dimensional spaces spanned by the orthonormal eigenfunctions of $H$ and $H_\sigma$. To be precise, we use the finite dimensional spaces $H^n:= \textrm{span}\{\psi_j \}_{j=1}^n$ and $H_\sigma^n:= \textrm{span}\{{\bv}_j \}_{j=1}^n$. We also utilize the orthogonal projections $P_\sigma^n:H_\sigma\to H_\sigma^n$ and $P^n:H\to H^n$ respectively defined as $P_\sigma^n {\bu} = \sum_{j=1}^n ({\bu},{\bv}_j)_\Omega {\bv}_j$ and $P^n\varphi = \sum_{j=1}^n (\varphi,\psi_j)_\Omega \psi_j$ for any ${\bu}\in H_\sigma$ and $\varphi\in H$. 

The projected variational problem now reads
\begin{subequations}
    \begin{align}
        \begin{aligned}
            &\langle \partial_t\varphi^n(t),\psi \rangle_V +  ({\bu}^n(t)\cdot\nabla\varphi^n(t),\psi)_\Omega + (\nabla\rho^n(t),\nabla\psi)_\Omega  \\ & = (\nabla(J\ast\varphi^n(t)),\nabla\psi)_\Omega - \ell_c(\nabla\theta^n(t),\nabla\psi)_\Omega,
        \end{aligned}\label{weakdisc:phi}
    \end{align}
    \begin{align}
        \begin{aligned}
            &\langle \partial_t{\bu}^n(t),{\bv} \rangle_{V_\sigma} + (({\bu}^n(t)\cdot\nabla){\bu}^n(t),{\bv} )_\Omega + 2(\nu(\varphi^n)\mathrm{D}{\bu}^n(t),\mathrm{D}{\bv}) \\& = \mathcal{K}(\bv\cdot\nabla\varphi^n(t), \mu^n(t) - \ell_c\theta^n(t) )_\Omega + (\ell(\varphi^n(t),\theta^n(t)){\bg},{\bv} )_\Omega + \langle{\bq}^n(t),{\bv})\rangle_{V_\sigma},
        \end{aligned}\label{weakdisc:u}
    \end{align}   
    \begin{align}
        \begin{aligned}
            & \langle \partial_t\theta^n(t),\vartheta  \rangle_V -  \ell_h\langle \partial_t\varphi^n(t),\vartheta \rangle_V +  \kappa(\nabla\theta^n(t),\nabla\vartheta)\\ & = ({\bu}^n(t)\cdot\nabla(\ell_h\varphi^n(t)-\theta^n(t)),\vartheta)_\Omega + ({\bg}\cdot{\bu}^n(t),\vartheta)_\Omega + \langle z^n(t),\vartheta\rangle_V
        \end{aligned}\label{weakdisc:theta}
    \end{align} 			
\end{subequations}
with $\rho(\cdot,\varphi^n) := a(\cdot)\varphi^n + F'(\varphi^n)$, $\mu^n= P^n(\rho(\cdot,\varphi^n) - J\ast \varphi^n + \ell_c\theta^n)$, $\varphi^n(0) = \varphi^n_0$, ${\bu}^n(0) = {\bu}_0^n$, and $\theta^n(0) = \theta^n_0$. Here the discrete initial conditions for the fluid velocity and temperature are obtained using the projections, i.e.,  ${\bu}_0^n = P_\sigma^n {\bu}_0$ and $\theta^n_0 = P^n\theta_0$, while the initial condition for order parameter is obtained as $\varphi^n_0 = (I + \frac{1}{n}B)^{-1}\varphi_0\in D(B)$.

Using the ansatz 
\begin{align*}
	\varphi^n(t) = \sum_{j=1}^{n} a^j(t)\psi_j, \quad {\bu}^n(t) = \sum_{j=1}^n b^j(t){\bv}_j, \quad \theta^n(t) = \sum_{j=1}^{n} c^j(t)\psi_j
\end{align*}
system \eqref{weakdisc:phi}-\eqref{weakdisc:theta} can be written as an ode involving the variables $\boldsymbol{a}(t):= [a^1(t),a^2(t),\ldots,a^n(t)]^\top$, $\boldsymbol{b}(t):= [b^1(t),b^2(t),\ldots,b^n(t)]^\top$, and $\boldsymbol{c}(t):= [c^1(t),c^2(t),\ldots,c^n(t)]^\top$ whose existence on the space $C^1([0,t^n];\mathbb{R}^n)$, for some $t^n\in (0,+\infty]$, can be established quite easily using, for example, Cauchy--Lipschitz theorem.

We now derive {\it a priori} estimates that establish $t^n=T$, and that the elements $\varphi^n$, ${\bu}^n$ and $\theta^n$ are uniformly bounded under appropriate norms. For the meantime, we shall drop the time variable $t$ for simplicity. Let us begin with taking $\psi = \mu^n(t)$ in \eqref{weakdisc:phi} to get 
\begin{align}
\begin{aligned}
	( \partial_t\varphi^n,\mu^n )_\Omega + &\, ({\bu}^n\cdot\nabla\varphi^n,\mu^n)_\Omega + (\nabla\rho^n,\nabla\mu^n)_\Omega = (\nabla(J\ast\varphi^n),\nabla\mu^n)_\Omega - \ell_c(\nabla\theta^n,\nabla\mu^n)_\Omega.
	\label{eq:1}
\end{aligned}
\end{align}
Here, we used the notation $\rho^n(t) := P^n\rho(\cdot,\varphi^n(t)) = \mu^n(t) + P^n(J\ast\varphi^n(t))-\ell_c P^n( \theta^n )$.

Furthermore, the first term on the left-hand side of \eqref{eq:1} can be written by using the explicit expression for $\mu^n$ and using the fact that $(\phi_1,J\ast\phi_2)_\Omega = (\phi_2,J\ast\phi_1)_\Omega$ from the property $J(x) = J(-x)$. Indeed, we have
\begin{align}
	\begin{aligned}
		&( \partial_t\varphi^n,\mu^n )_\Omega = (\partial_t\varphi^n, a\varphi^n + F'(\varphi^n) - J\ast \varphi^n + \ell_c\theta^n)_\Omega\\
		&= \frac{1}{2}\frac{d}{dt}\left(\int_\Omega a|\varphi^n(x)|^2 \du x + 2\int_\Omega F(\varphi^n(x)) \du x  - \int_\Omega \varphi^n(x) (J\ast\varphi^n(x)) \du x\right) + \ell_c( \partial_t\varphi^n,\theta^n )_\Omega\\
		&= \frac{1}{2}\frac{d}{dt}\left( \mathbb{E}_{nl}(\varphi^n)
            + 2 \int_\Omega F(\varphi^n(x))\du x \right) + \ell_c( \partial_t\varphi^n,\theta^n )_\Omega.
	\end{aligned}\label{eq:2}
\end{align}
Using the definition of $\rho^n$, and \eqref{eq:2} we rewrite \eqref{eq:1} as
\begin{align}
\begin{aligned}
		&\frac{1}{2}\frac{d}{dt}\left( \mathbb{E}_{nl}(\varphi^n) + 2 \int_\Omega F(\varphi^n(x))\du x \right)  + \|\nabla\mu^n\|^2\\ & = ({\bu}^n\cdot\nabla\mu^n,\varphi^n)_\Omega + (\nabla J\ast\varphi^n,\nabla\mu^n)_\Omega - (\nabla P^n(J\ast\varphi^n),\nabla\mu^n)_\Omega -\ell_c( \partial_t\varphi^n,\theta^n )_\Omega.
	\end{aligned}\label{energyeq:phi}
\end{align}

Let us now test \eqref{weakdisc:u} with ${\bv} = {\bu}^n(t)$ to get 
\begin{align}
	\begin{aligned}
		&\frac{1}{2}\frac{d}{dt}\|{\bu}^n\|^2 +  2\|\sqrt{\nu(\varphi)}\mathrm{D}{\bu}^n\|^2 = \mathcal{K}({\bu}^n\cdot\nabla\varphi^n, \mu^n - \ell_c\theta^n )_\Omega + (\ell(\varphi^n,\theta^n){\bg},{\bu}^n )_\Omega+ \langle{\bq}^n,{\bu}^n\rangle_{V_\sigma}
	\end{aligned}\label{energyeq:u}
\end{align}
Similarly, substituting $\vartheta = \theta^n(t)$ into \eqref{weakdisc:theta} gives us 
\begin{align}
\begin{aligned}
&\frac{1}{2}\frac{d}{dt}\|\theta^n\|^2 + \kappa\|\nabla\theta^n\|^2 =  \ell_h (\partial_t\varphi^n,\theta^n)_\Omega + \ell_h({\bu}^n\cdot\nabla\varphi^n,\theta^n)_\Omega + ({\bg}\cdot{\bu}^n,\theta^n)_\Omega + \langle z^n,\theta^n\rangle_{V}
\end{aligned}\label{energyeq:theta}
\end{align}

To reconcile the computations above we divide the equations \eqref{energyeq:phi}, \eqref{energyeq:u} and \eqref{energyeq:theta} by the constants $\ell_c$, $\mathcal{K}\ell_c$ and $\ell_h$, respectively, and add the results together. Such process yields
\begin{align}
\begin{aligned}
	&\frac{d}{dt}\mathbb{E}(\varphi^n(t),{\bu}^n(t),\theta^n(t)) + \mathbb{D}(\varphi^n(t),\mu^n(t),{\bu}^n(t),\theta^n(t)) \\ 
	&=  \frac{1}{\ell_c}\left\{ (\nabla J\ast\varphi^n(t),\nabla\mu^n(t))_\Omega - (\nabla P^n(J\ast\varphi^n(t)),\nabla\mu^n(t))_\Omega \right\} \\
	&\ \ \  + \frac{1}{\mathcal{K}\ell_c}\left\{(\ell(\varphi^n(t),\theta^n(t)){\bg},{\bu}^n(t) )_\Omega + \langle{\bq}^n(t),{\bu}^n(t)\rangle_{V_\sigma}\right\}\\
	 &\ \ \ +  \frac{1}{\ell_h}\left\{({\bg}\cdot{\bu}^n(t),\theta^n(t))_\Omega + \langle z^n(t),\theta^n(t)\rangle_{V}\right\}
\end{aligned}\label{eqreconcile}
\end{align}

The next step is to gain some estimates for the right-hand side of \eqref{eqreconcile}. Let us begin with the terms multiplied with $1/\ell_c$:
\begin{align}
\begin{aligned}
    |(&\nabla J\ast\varphi^n(t),\nabla\mu^n(t))_\Omega  - (\nabla P^n(J\ast\varphi^n(t)),\nabla\mu^n(t))_\Omega|\\
    &\le  \|\nabla J\|_{L^1}\|\varphi^n(t)\|\|\nabla\mu^n(t)\| + \|\nabla P^n(J\ast\varphi^n(t)) \|\|\nabla\mu^n(t)\|   \\
    & \le  \|\nabla J\|_{L^1}\|\varphi^n(t)\|\|\nabla\mu^n(t)\| + \|B^{1/2} P^n(J\ast\varphi^n(t)) \|\|\nabla\mu^n(t)\|  \\
    &  \le \|\nabla J\|_{L^1}\|\varphi^n(t)\|\|\nabla\mu^n(t)\| + \|\nabla J\ast\varphi^n(t) \|\|\nabla\mu^n(t)\|  \\
    &  \le  c\|J\|_{W^{1,1}}\|\varphi^n(t)\|\|\nabla\mu^n(t)\| \le \frac{1}{2}\|\nabla\mu^n(t)\|^2 + \frac{c}{2}\|J\|_{W^{1,1}}^2\|\varphi^n(t)\|^2
\end{aligned}\label{ineqlc}
\end{align}
In the computations above, we utilized H{\"o}lder's inequality as well as Young's convolution inequality. The last estimate, on the other hand is achieved using Young's inequality.
Similarly, the terms with $1/\mathcal{K}\ell_c$ are estimated with the H{\"o}lder's and Young's inequalities, and the constancy of the gravity $\bg$. The computation is shown below:
\begin{align}
    \begin{aligned}
        | (\ell(&\varphi^n(t) ,\theta^n(t)){\bg},{\bu}^n(t) )_\Omega + \langle {\bq}^n(t),{\bu}^n(t)\rangle_{V_\sigma}|\\
        & \le \|\ell(\varphi^n(t),\theta^n(t)){\bg}\| \|{\bu}^n(t)\| + \| {\bq}^n(t)\|_{V_\sigma^*}\|{\bu}^n(t)\|_{V_\sigma}\\
        & \le c\left( 1 + \frac{\alpha\mathcal{K}}{2}\|\varphi^n(t)\|^2 + \frac{\mathcal{K}\ell_c}{4\ell_h}\|\theta^n(t)\|^2  + \frac{1}{4}\|{\bu}^n(t)\|^2 \right) + \frac{c_{\rm K}}{4\underline{\nu}}\| {\bq}^n(t)\|_{V_\sigma^*}^2 + \frac{\underline{\nu}}{c_{\rm K}}\|{\bu}^n(t)\|_{V_\sigma}^2,
    \end{aligned}\label{ineqKlc}
\end{align}
where $c>0$ depends on $\mathcal{K}$, $\ell_h$, $\ell_c$, $\alpha$ (which is defined as $\alpha:= 2c_1 - \|J\|_{L^1} > 0$ based on assumption (A4)) and $\bg$, the constant $c_{\rm K}$ on the other hand corresponds to that in Korn's inequality.  Lastly, we employ the same techniques we previously utilized to the terms multiplied with $1/\ell_h$
\begin{align}
    \begin{aligned}
        |({\bg}&\cdot{\bu}^n(t),\theta^n(t))_\Omega + \langle z^n(t),\theta^n(t)\rangle_{V}|\\ & \le c\left( \frac{1}{4}\|\theta^n(t)\|^2  + \frac{\ell_h}{4\mathcal{K}\ell_c}\|{\bu}^n(t)\|^2 \right) + \frac{1}{2\kappa}\|z^n(t)\|_{V^*}^2 + \frac{\kappa}{2}\|\theta^n(t)\|_V^2.
    \end{aligned}\label{ineqlh}
\end{align}
Plugging estimates \eqref{ineqlc}, \eqref{ineqKlc} and \eqref{ineqlh} into \eqref{eqreconcile}, and using the lower bound for the viscosity and Korn's inequality will yield
\begin{align}
    \begin{aligned}
        &\frac{d}{dt}\mathbb{E}(\varphi^n(t),{\bu}^n(t),\theta^n(t)) + \widehat{\mathbb{D}}(\mu^n(t),{\bu}^n(t),\theta^n(t)) \\
        & \le c\left( 1 + \frac{\alpha}{2\ell_c}\|\varphi^n(t)\|^2 + \frac{1}{2\ell_h}\|\theta^n(t)\|^2  + \frac{1}{2\mathcal{K}\ell_c}\|{\bu}^n(t)\|^2 \right)   \\
        &\ \ \ + \frac{c}{2}\|J\|_{W^{1,1}}^2\|\varphi^n(t)\|^2 + \frac{c_{\rm K}}{4\underline{\nu}\mathcal{K}\ell_c}\| {\bq}^n(t)\|_{V_\sigma^*}^2 + \frac{1}{2\kappa\ell_h}\|z^n(t)\|_{V^*}^2,
    \end{aligned}\label{energy:incomplete}
\end{align}
where $\widehat{\mathbb{D}}(\mu,{\bu},\theta) := \frac{1}{2\ell_c}\|\nabla\mu\|^2 + \frac{\underline{\nu}}{c_{\rm K}\mathcal{K}\ell_c}\|{\bu}\|_{V_\sigma}^2 + \frac{\kappa}{2\ell_h}\|\nabla\theta\|^2 $.

To this end, let us note that the contribution of $\varphi^n$ in the total energy functional $\mathbb{E}$ can be estimated from below. First, we can rewrite such expression using \eqref{defn:aundconv} and bound it below using Young's convolution inequality:
\begin{align*}
\begin{aligned}
    \frac{1}{2}&\int_\Omega\int_\Omega J(x-y)(\varphi^n(x) - \varphi^n(y))^2 \du x\du y + 2\int_\Omega F(\varphi^n(x)) \du x\\
    & = \int_\Omega a(x)\varphi^n(x)^2\du x - \int_\Omega (J\ast \varphi^n)(x)\varphi^n(x)\du x + 2 \int_\Omega F(\varphi^n(x)) \du x\\
    & \ge \int_\Omega a(x)\varphi^n(x)^2\du x - \|J\|_{L^1}\int_\Omega \varphi^n(x)^2 \du x + 2 \int_\Omega F(\varphi^n(x)) \du x.
\end{aligned}
\end{align*}
From assumption (A4) and the nonegativity of $a(\cdot)$ we further minorize the left-hand side of the expression above:
\begin{align}
\begin{aligned}
    \frac{1}{2}&\int_\Omega\int_\Omega J(x-y)(\varphi^n(x) - \varphi^n(y))^2 \du x\du y + 2\int_\Omega F(\varphi^n(x)) \du x\\
    & \ge \int_\Omega a(x)\varphi^n(x)^2\du x - \|J\|_{L^1}\int_\Omega \varphi^n(x)^2 \du x + 2 \int_\Omega F(\varphi^n(x)) \du x\\
    & \ge  (2c_1- \|J\|_{L^1} )\int_\Omega \varphi^n(x)^2 \du x - 2c_2|\Omega| =:  \alpha \|\varphi^n\|^2 - c.
\end{aligned}\label{phicomponentE}
\end{align}

Integrating \eqref{energy:incomplete} over $(0,t)$ for $t\in (0,t^n]$ and utilizing \eqref{phicomponentE} -- after some rearrangement -- gives us
\begin{align*}
    \begin{aligned}
        &\widehat{\mathbb{E}}(\varphi^n(t),{\bu}^n(t),\theta^n(t))+ \int_0^t \widehat{\mathbb{D}}(\mu^n(s),{\bu}^n(s),\theta^n(s))\du s \\
        & \le \mathbb{E}(\varphi^n(0),{\bu}^n(0),\theta^n(0)) + \frac{c_{\rm K}}{4\underline{\nu}\mathcal{K}\ell_c}\int_0^t \| {\bq}^n(s\|_{V_\sigma^*}^2\du s + \frac{1}{2\kappa\ell_h}\int_0^t \|z^n(s)\|_{V^*}^2 \du s  \\
        &\ \ \  + c\left( 1 + \int_0^t \left\{\frac{c}{2}\|J\|_{W^{1,1}}^2\|\varphi^n(s)\|^2 + \widehat{\mathbb{E}}(\varphi^n(s),\theta^n(s),{\bu}^n(s)) \right\}\du s\right) ,
    \end{aligned}
\end{align*}
where $\widehat{\mathbb{E}}(\varphi,{\bu},\theta) := \frac{\alpha}{2\ell_c}\|\varphi\|^2 + \frac{1}{2\ell_h}\|\theta\|^2  + \frac{1}{2\mathcal{K}\ell_c}\|{\bu}\|^2$. Gr{\"o}nwall's inequality assures us of the existence of a constant $c>0$ such that 
\begin{align}
    \begin{aligned}
        &\sup_{t\in (0,t^n]}\widehat{\mathbb{E}}(\varphi^n(t),{\bu}^n(t),\theta^n(t))+ \int_0^{t^n} \widehat{\mathbb{D}}(\mu^n(s),{\bu}^n(s),\theta^n(s)) \du s \\
        & \le ce^{cT}\left( 1 + \mathbb{E}(\varphi^n(0),{\bu}^n(0),\theta^n(0))  + \| {\bq}\|_{L^2(I;V_\sigma^*)} +\|z\|_{L^2(I;V^*)}\right).
    \end{aligned}\label{estimate:completeN}
\end{align}
Note that the constant $c>0$ above is independent of $n$, so that further estimate on $\mathbb{E}(\varphi^n(0),{\bu}^n(0),\theta^n(0))$ with independence on $n$ will yield uniform boundedness of the left-hand side of \eqref{estimate:completeN} which implies 
\begin{align}
    &\|\varphi^n \|_{L^\infty(I;H)} \le M\label{uniformbounds:phiH}\\ 
    &\|\nabla\mu^n \|_{L^2(I;L^2(\Omega)^2)} \le M\label{uniformbounds:mul2}\\ 
    &\|{\bu}^n \|_{L^\infty(I;H_\sigma)\cap L^2(I;V_\sigma)} \le M\label{uniformbounds:uHV}\\
    &\|{\theta}^n \|_{L^\infty(I;H)} \le M.\label{uniformbounds:thetaH}\\
    &\|\nabla{\theta}^n \|_{L^2(I;L^2(\Omega)^2)} \le M.\label{uniformbounds:thetaH1}
\end{align}
Indeed, from the definition of the orthogonal projections $P^n$ and $P^n_\sigma$ whose norms are bounded by 1 in the linear spaces $\mathcal{L}(H,H)$ and $\mathcal{L}(H_\sigma,H_\sigma)$, respectively, we see that $\|{\bu}^n(0)\|_{H_\sigma} \le \|{\bu}_0\|_{H_\sigma}$ and $\|\theta^n(0) \|_H\le \|\theta_0\|_H$. 
For the contribution of the order parameter, we first note that the Neumann operator $B$ is a maximal monotone operator hence $\varphi^n_0 \to \varphi_0$ in $H$. Hence, there exists $c>0$ independent of $n$ such that
\begin{align*}
   &\mathbb{E}_{nl}(\varphi^n_0)+ 2\int_\Omega F(\varphi^n_0(x))\du x \le c \|J\|_{L^{1}}\|\varphi^n_0\|^2 + 2\int_\Omega F(\varphi^n_0(x))\du x\\
    &\le c \|J\|_{L^{1}}\|\varphi_0\|^2 + 2\int_\Omega F(\varphi^n_0(x))\du x.
\end{align*}
The next challenge lies on controlling the nonlinear part by the initial data. As previously remarked, we can write $F$ -- by virtue of (A3) -- as $F(s) = G(s) - \frac{a^*}{2}s^2$, hence
\begin{align*}
    \int_\Omega F(\varphi^n_0(x))\du x = \int_\Omega G(\varphi^n_0(x))\du x - \frac{a^*}{2}\|\varphi^n_0\|^2 .
\end{align*}
Notably the function $h = G'$ is monotonically increasing which we can suppose to satisfy $h(0) = 0$. If we multiply $h(\varphi^n_0)$ by $-(\varphi_0 - \varphi_0^n) = -\frac{1}{n}B\varphi_0^n$ and integrate over $\Omega$ one gets
\begin{align*}
    \begin{aligned}
        \int_\Omega h(\varphi^n_0)(\varphi_0^n - \varphi_0) \du x & = -\frac{1}{n}\int_\Omega h(\varphi^n_0)B\varphi_0^n \du x = -\frac{1}{n}\int_\Omega \left\{ \nabla h(\varphi^n_0) \cdot \nabla \varphi_0^n  \right\} \du x\\
        & = -\frac{1}{n}\int_\Omega \left\{ h'(\varphi^n_0) |\nabla \varphi_0^n|^2\right\} \du x \le 0.
    \end{aligned}
\end{align*}
Expanding $G$ about $\varphi^n_0$, the convexity of $G$ then implies
\begin{align*}
    \begin{aligned}
        \int_\Omega G(\varphi^n_0) \du x \le \int_\Omega G(\varphi_0) + h(\varphi^n_0) (\varphi_0^n - \varphi_0) \du x \le \int_\Omega G(\varphi_0) \du x.
    \end{aligned}
\end{align*}
Using Fatou's lemma, we get 
\begin{align*}
    \limsup_{n\to \infty}\int_\Omega F(\varphi^n_0) \du x \le \int_\Omega G(\varphi_0) \du x - \frac{a^*}{2}\|\varphi_0\|^2 = \int_\Omega F(\varphi_0) \du x .
\end{align*}
In summary, we get the energy estimate
\begin{align*}
    \begin{aligned}
        &\sup_{t\in [0,T]}\widehat{\mathbb{E}}(\varphi^n(t),{\bu}^n(t),\theta^n(t))+ \int_0^{T} \widehat{\mathbb{D}}(\mu^n(s),{\bu}^n(s),\theta^n(s)) \du s \\
        & \le ce^{cT}\left( 1 + \mathbb{E}(\varphi_0,{\bu}_0,\theta_0) + \frac{c_{\rm K}}{4\underline{\nu}\mathcal{K}\ell_c}\| {\bq}\|_{L^2(I;V_\sigma^*)} + \frac{1}{2\kappa\ell_h}\|z\|_{L^2(I;V^*)}\right)
    \end{aligned}
\end{align*}
which validates \eqref{uniformbounds:phiH}--\eqref{uniformbounds:thetaH1}.

On the other hand, if we test \eqref{weakdisc:theta} with $\vartheta = 1$ we get 
\begin{align*}
    \left| \int_\Omega \theta^n(x)\du x \right| \le c\left( \|\theta_0\|_H + \|{\bu^n}\|_{L^\infty(I;H_\sigma)} + \|z\|_{L^2(I;V^*)} \right).
\end{align*}
The inequality above, together with \eqref{uniformbounds:uHV} and \eqref{uniformbounds:thetaH1}, imply
\begin{align}
 &\|{\theta}^n \|_{L^\infty(I;H)\cap L^2(I;V)} \le M.\label{uniformbounds:thetaHV}   
\end{align}

At this moment, we shall derive some estimates for $\varphi^n$ in $L^2(I;V)$. First, we use the definition of $\mu^n$, assumption (A3), and Young's convolution and Young's inequalities to get
\begin{align}
    \begin{aligned}
        &(\nabla\varphi^n,\nabla\mu^n)_\Omega\\ & = (\nabla\varphi^n, (F''(\varphi^n) + a)\nabla\varphi^n +\varphi^n\nabla a  - \nabla J\ast\varphi^n + \ell_c\nabla\theta^n )_\Omega\\
        & \ge c_0\|\nabla\varphi^n\|^2 - 2\|J\|_{W^{1,1}}\|\nabla\varphi^n\|\|\varphi^n\| + \ell_c(\nabla\varphi^n,\nabla\theta^n)\\
        &\ge \frac{3c_0}{4}\|\nabla\varphi^n\|^2 - \frac{4}{c_0}\|J\|_{W^{1,1}}\|\varphi^n\|^2 + \ell_c(\nabla\varphi^n,\nabla\theta^n)
    \end{aligned}\label{estimate:phiV}
\end{align}
Using H{\"o}lder's and Young's inequalities on $(\nabla\varphi^n,\nabla\mu^n)_\Omega$, on the other hand, one gets $(\nabla\varphi^n,\nabla\mu^n)_\Omega \le (c_0/4)\|\nabla\varphi^n\|^2 + (1/c_0)\|\nabla\mu^n\|^2$. Combining this with \eqref{estimate:phiV} we have
\begin{align}
    \begin{aligned}
        & \frac{c_0}{2}\|\nabla\varphi^n\|^2 \le \left| \frac{1}{c_0}\|\nabla\mu^n\|^2 + \frac{4}{c_0}\|J\|_{W^{1,1}}\|\varphi^n\|^2 - \ell_c(\nabla\varphi^n,\nabla\theta^n)\right|\\
        & \le \frac{1}{c_0}\|\nabla\mu^n\|^2 + \frac{4}{c_0}\|J\|_{W^{1,1}}\|\varphi^n\|^2 + \frac{\ell_c^2}{c_0}\|\nabla\theta^n\| + \frac{c_0}{4}\|\nabla\varphi^n\|^2.
    \end{aligned}\label{estimate:nablaphi}
\end{align}
Moving the last term on the right-hand side of \eqref{estimate:nablaphi} to the other side, utilizing the total mass conservation of the order parameter, and the inequalities \eqref{uniformbounds:phiH}, \eqref{uniformbounds:mul2} and \eqref{uniformbounds:thetaHV} we get the uniform boundedness of $\varphi^n$ in ${L^2(I;V)}$, i.e.,
\begin{align}
    \|\varphi^n\|_{L^2(I;V)} \le M.\label{estimate:phiVV}
\end{align} 
To strengthen \eqref{uniformbounds:mul2} let us prove that the average of $\mu^n$ -- which we denote as $\widehat{\mu^n}:= \frac{1}{|\Omega|}\int_\Omega \mu^n\du x$-- over $\Omega$ is uniformly bounded as well.
Indeed, by the definition of $\mu^n$ and by assumption (A5) we have
\begin{align*}
    \begin{aligned}
        \left|\int_\Omega \mu^n \du x \right| & = \left|\int_\Omega P^n(a\varphi^n - J\ast\varphi^n + F'(\varphi^n) + \ell_c\theta^n) \du x \right|\\
        & \le \int_\Omega |F'(\varphi^n)| \du x + \ell_c\sqrt{|\Omega|}\|\theta^n\|\\
        & \le c_3 \int_\Omega |F(\varphi^n)| \du x + c_4|\Omega| + \ell_c\sqrt{|\Omega|}\|\theta^n\|.
    \end{aligned}
\end{align*}
Here we used the fact that $\int_\Omega P^n(a\varphi^n - J\ast\varphi^n)\du x = 0$. To close the estimate above we note that $\|F(\varphi^n) \|_{L^1}$ can be shown to be uniformly bounded by integrating \eqref{energy:incomplete} over $(0,T)$, using Gronwall's inequality and using the upper bound for $\mathbb{E}(\varphi^n(0),{\bu}^n(0),\theta^n(0))$. The boundedness above, together with \eqref{uniformbounds:mul2} and Poincar{\'e}-Wirtinger inequality implies
\begin{align}
    \|\mu^n \|_{L^2(I;V)} \le M. \label{boundedness:muV}
\end{align}
Before we move on, we mention -- by virtue of assumption (A5) -- that $\|\rho(\cdot,\varphi^n)\|_{L^\infty(I;L^p(\Omega))}$ is uniformly bounded where $p\in(1,2]$ is as in the aforementioned assumption. Indeed, we have
\begin{align*}
    \begin{aligned}
        \|\rho(\cdot,\varphi^n)\|_{L^p} & \le \|a\|_{L^\infty}\|\varphi^n\|_{L^p} + \left( \int_\Omega |F'(\varphi^n(x))|^p \du x \right)^{1/p}\\
        & \le \|a\|_{L^\infty}\|\varphi^n\| + \left( \int_\Omega c_3|F(\varphi^n(x))| + c_4 \du x \right)^{1/p}\\
        & \le \|a\|_{L^\infty}\|\varphi^n\| + c_3\|F(\varphi^n)\|_{L^1}^{1/p} + c_4{|\Omega|}^{1/p}.
    \end{aligned}
\end{align*}
The uniform boundedness follows from the boundedness of $\|F(\varphi^n)\|_{L^1}$ and estimate \eqref{uniformbounds:phiH}, i.e.
\begin{align}
    \|\rho(\cdot,\varphi^n)\|_{L^\infty(I;L^p(\Omega))} \le M.\label{estimate:rho}
\end{align}

Our next aim is to establish some boundedness for the time differentiated variables. Let us begin with the time derivative of the fluid velocity ${\bu}^n$. We rewrite \eqref{weakdisc:u} as 
\begin{align}
    \begin{aligned}
\langle& \partial_t{\bu}^n(t),{\bv}\rangle_{V_\sigma} +  \langle P^n_\sigma\mathcal{B}(\bu^n(t)),{\bv} \rangle_{V_\sigma} + \langle P^n_\sigma\mathcal{A}({\bu}^n(t),\varphi^n(t)),\bv\rangle_{V_\sigma}\\ & = \mathcal{K}\langle P^n_\sigma\mathcal{C}_2(\varphi^n(t),\mu^n(t) - \ell_c\theta^n(t)),{\bv} \rangle_{V_\sigma} + (P^n_\sigma\ell(\varphi^n(t),\theta^n(t)){\bg},{\bv} )_\Omega + \langle P^n_\sigma{\bq}(t),{\bv})\rangle_{V_\sigma}.
\end{aligned}\label{projectedeq:u}
\end{align}
Using the antisymmetry of $b({\bu};\cdot,\cdot)$ for ${\bu}\in V_\sigma$, and the estimates \eqref{trilinholder3d} and \eqref{trilinholder2d} we get
\begin{align*}
    \begin{aligned}
        \| P^n_\sigma\mathcal{B}(\bu^n(t))\|_{V_\sigma^*}&\le  \inf_{\substack{{\bv}\in V_\sigma\\ \|{\bv}\|_{V_\sigma} = 1}} |b(\bu^n(t);\bu^n(t),{\bv})| \\
        &\le \inf_{\substack{{\bv}\in V_\sigma\\ \|{\bv}\|_{V_\sigma} = 1}}  c\left\{\begin{matrix}
    \|\bu^n(t)\|_{H_\sigma}^{1/2} \|\bu^n(t)\|_{V_\sigma}^{3/2} \|\bv\|_{V_\sigma}& \text{for }d = 3\\
    \|\bu^n(t)\|_{H_\sigma} \|\bu^n(t)\|_{V_\sigma}  \|\bv\|_{V_\sigma}& \text{for }d = 2
        \end{matrix} \right.\\
        &\le  c\left\{\begin{matrix}
    \|\bu^n(t)\|_{H_\sigma}^{1/2} \|\bu^n(t)\|_{V_\sigma}^{3/2} & \text{for }d = 3\\
    \|\bu^n(t)\|_{H_\sigma} \|\bu^n(t)\|_{V_\sigma}  & \text{for }d = 2.
        \end{matrix} \right.
    \end{aligned}
\end{align*}
For $d=3$, we thus see that
\begin{align}
    \| P^n_\sigma\mathcal{B}({\bu}^n)\|_{L^{4/3}(I;V_\sigma^*)} \le \|{\bu}^n\|_{L^\infty(I;H_\sigma)}^{1/2}\|{\bu}^n\|_{L^2(I;V_\sigma)}^{3/2}\le M^2, \label{estimate:trilinD3}
\end{align}
while when $d=2$ we have 
\begin{align}
        \| P^n_\sigma\mathcal{B}({\bu}^n)\|_{L^{2}(I;V_\sigma^*)} \le \|{\bu}^n\|_{L^\infty(I;H_\sigma)}\|{\bu}^n\|_{L^2(I;V_\sigma)}\le M^2.\label{estimate:trilinD2}
\end{align}
For the viscosity term we use H{\"o}lder's inequality and assumption (A2) for both $d=2,3$ to get
\begin{align}
\begin{aligned}
    \| P^n_\sigma\mathcal{A}({\bu}^n(t),\varphi^n(t))\|_{V_\sigma^*}  = \inf_{\substack{{\bv}\in V_\sigma\\ \|{\bv}\|_{V_\sigma} = 1}} 2(\nu(\varphi^n(t))\mathrm{D}{\bu}^n(t),\mathrm{D}{\bv})_{\Omega}& \\
    \le 2 \inf_{\substack{{\bv}\in V_\sigma\\ \|{\bv}\|_{V_\sigma} = 1}} \|\nu(\varphi^n(t))\mathrm{D}{\bu}^n(t) \| \|\mathrm{D}{\bv}\| \le 2\overline{\nu} \|\bu^n(t)\|_{V_\sigma}. &
\end{aligned}\label{estimate:dissipation}
\end{align}
This implies that $\| P^n_\sigma\mathcal{A}({\bu}^n,\varphi^n)\|_{L^2(I;V_\sigma^*)}\le M$. 

For the term that takes into account the capillarity we first deal with the case $d=3$ as follows:
\begin{align}
    \begin{aligned}
        \| P^n_\sigma \mathcal{C}_2(\varphi^n(t),\mu^n(t) - \ell_c\theta^n(t)) \|_{V_\sigma^*}
        &= \inf_{\substack{{\bv}\in V_\sigma\\ \|{\bv}\|_{V_\sigma} = 1}} | \langle P^n_\sigma\mathcal{C}_2(\ell_c\theta^n(t) -\mu^n(t),\varphi^n(t)),{\bv} \rangle_{V_\sigma}|\\
        & = \inf_{\substack{{\bv}\in V_\sigma\\ \|{\bv}\|_{V_\sigma} = 1}} |(\varphi^n(t)\nabla(\mu^n(t) - \ell_c\theta^n(t)),{\bv} )_\Omega|\\
        & \le \inf_{\substack{{\bv}\in V_\sigma\\ \|{\bv}\|_{V_\sigma} = 1}} \|\varphi^n(t)\|_{L^3}\|\nabla(\mu^n(t) - \ell_c\theta^n(t))\|\|{\bv}\|_{L^6} \\
        & \le \inf_{\substack{{\bv}\in V_\sigma\\ \|{\bv}\|_{V_\sigma} = 1}} \|\varphi^n(t)\|_{L^3}\|\nabla(\mu^n(t) - \ell_c\theta^n(t))\|\|{\bv}\|_{V_\sigma} .
    \end{aligned}\label{capillaritylatence}
\end{align}
We note that to be able to reach the third line, we used H{\"o}lder's inequality, while Rellich-Kondrachov embedding theorem --- which ensures us of the continuous embeddings $W^{1,2}(\Omega)\hookrightarrow L^{6}(\Omega)$ --- to get the last inequality. Now, the interpolation of $L^p$ spaces, with $\theta = \frac{1}{2}$, $p_1=6$ and $p_2 = 2$, gives us \begin{align}
\|\varphi\|_{L^3}\le c\|\varphi\|^{1/2}_{L^6}\|\varphi\|^{1/2}.\label{interpolate}
\end{align}
Using \eqref{interpolate} to \eqref{capillaritylatence}, and the continuous embedding $W^{1,2}(\Omega)\hookrightarrow L^{6}(\Omega)$ lead to
\begin{align*}
    \begin{aligned}
        \| P^n_\sigma\mathcal{C}_2(\varphi^n(t),\mu^n(t) - \ell_c\theta^n(t)) \|_{V_\sigma^*}
        & \le  c\|\varphi^n(t)\|^{1/2}\|\varphi^n(t)\|^{1/2}_{L^6}\|\nabla(\mu^n(t) - \ell_c\theta^n(t))\| \\
        & \le  c\|\varphi^n(t)\|^{1/2}\|\varphi^n(t)\|_{V}^{1/2}\|\nabla(\mu^n(t) - \ell_c\theta^n(t))\| .
    \end{aligned}
\end{align*}
For the next computations, we employ Young's inequality $\alpha\beta \le \frac{\varepsilon}{p}\alpha^p_1 + \frac{1}{p_2\varepsilon^{p'/p}}\beta^{p'}$, which holds for any $\varepsilon,\alpha,\beta>0$ and $1/p_1 + 1/p_2 = 1$. Taking $p_1 = 3$, $p_2=3/2$, $\alpha = \|\varphi^n(t)\|_{V}^{2/3}$, and $\beta = \|\nabla(\mu^n(t) - \ell_c\theta^n(t))\|^{4/3}$ and from the estimates \eqref{uniformbounds:phiH}, \eqref{uniformbounds:mul2}, \eqref{uniformbounds:thetaHV} and \eqref{estimate:phiVV} we get the following estimates for the three dimensional case:
\begin{align}
    \begin{aligned}
        &\| P^n_\sigma\mathcal{C}_2(\varphi^n,\mu^n - \ell_c\theta^n) \|_{L^{4/3}(I;V_\sigma^*)}\\ & = c\left( \int_0^T \left( \|\varphi^n(t)\|^{\frac{1}{2}}\|\varphi^n(t)\|_{V}^{\frac{1}{2}}\|\nabla(\mu^n(t) - \ell_c\theta^n(t))\|\right)^{\frac{4}{3}} \du t \right)^{\frac{3}{4}}\\
        & \le c\|\varphi^n\|_{L^\infty(I;H)}^{\frac{1}{2}}\left( \int_0^T \|\varphi^n(t)\|_{V}^{\frac{2}{3}}\|\nabla(\mu^n(t) - \ell_c\theta^n(t))\|^{\frac{4}{3}}\du t \right)^{\frac{3}{4}}\\
        & \le c\|\varphi^n\|_{L^\infty(I;H)}^{\frac{1}{2}}\left\{ \|\varphi^n \|_{L^2(I;V)}^{\frac{3}{2}} + \|\nabla(\mu^n - \ell_c\theta^n(t) )\|_{L^2(I;L^2(\Omega)^2)}^{\frac{3}{2}}\right\}\\
        & \le M^2.
    \end{aligned}\label{estimate:capillarityD3}
\end{align}

For the two dimensional case, we rewrite $(\mu^n-\ell_c\theta^n)\nabla\varphi^n$ as
\begin{align}
	(\mu^n-\ell_c\theta^n)\nabla\varphi^n = \nabla\left( F(\varphi^n) + a(\cdot)\frac{(\varphi^n)^2}{2} \right) - \nabla a(\cdot) \frac{(\varphi^n)^2}{2} - (J\ast\varphi^n)\nabla\varphi^n.
\end{align}
This allows us to have
\begin{align}
	\begin{aligned}
		&\int_0^T((\mu^n(t)-\ell_c\theta^n(t))\nabla\varphi^n(t),\bv(t))_\Omega \du t\\ & \le \int_0^T |(\nabla a \frac{(\varphi(t)^n)^2}{2},\bv(t))_\Omega|\du t + \int_0^T |((\nabla J\ast\varphi^n(t))\varphi^n(t),\bv(t))_\Omega|\du t\\
		&\le c(\|\nabla a\|_{L^\infty}\| + \|\nabla J\|_{L^1})\int_0^T \|\varphi^n(t)\|^2_{L^4}\|\bv(t)\|_{H_\sigma}\du t\\
		& \le c(\|\nabla a\|_{L^\infty}\| + \|\nabla J\|_{L^1})\|\varphi^n\|^2_{L^4(I;L^4(\Omega))}\|\bv\|_{L^2(I;V_\sigma)}.
	\end{aligned}
\end{align}
From Gagliardo-Nirenberg inequality, \eqref{uniformbounds:phiH} and \eqref{estimate:phiVV}, we infer that
\begin{align}
	\begin{aligned}
		\|\varphi^n\|_{L^4(I;L^4(\Omega))}^4 & \le \int_0^T \|\varphi^n(t)\|_{L^2}^2\|\nabla\varphi^n(t)\|_{L^2}^2 \du t \le \|\varphi^n\|_{L^\infty(I;H)}^2\|\varphi^n\|_{L^2(I;V)}^2 \le M^4.
	\end{aligned}
\end{align}
Therefore,
\begin{align}\label{estimate:capillarityD2}
	\| P^n_\sigma\mathcal{C}_2(\varphi^n,\mu^n - \ell_c\theta^n) \|_{L^{2}(I;V_\sigma^*)} \le M^2.
\end{align}


The remaining terms in \eqref{projectedeq:u} are trivially handled and it can be easily shown that 
    \begin{align}
        &\|P^n_\sigma\ell(\varphi^n(t),\theta^n(t)){\bg} \|_{L^2(I;V_\sigma^*)} \le 1+2M,\label{estimate:ell}\\
        &\|P^n_\sigma {\bq}(t) \|_{L^2(I;V_\sigma^*)} \le 1+\|{\bq}(t) \|_{L^2(I;V_\sigma^*)} ,\label{estimate:q}
    \end{align}
where \eqref{estimate:ell} is achieved from \eqref{uniformbounds:phiH} and \eqref{uniformbounds:thetaHV}, while \eqref{estimate:q} is owed from the fact that $P^n_\sigma\in\mathcal{L}(V_\sigma^*,V_\sigma^*)$. Combining the estimates \eqref{estimate:trilinD3}, \eqref{estimate:trilinD2}, \eqref{estimate:dissipation}, \eqref{estimate:capillarityD3}, \eqref{estimate:capillarityD2}, \eqref{estimate:ell}, and \eqref{estimate:q}, we get the estimate
\begin{align}
    \|\partial_t{\bu}^n \|_{L^q(I;V_\sigma^*)}\le c(1+(1+M)M +\|{\bq}(t) \|_{L^2(I;V_\sigma^*)}) 
    \label{estimate:timeu}
\end{align}
for some constant $c>0$, and where $q=4/3$ if $d=3$, and $q=2$ if $d=2$.

Now let us derive estimates for the time derivative of the order parameter. To do this we first rewrite \eqref{weakdisc:phi} 
\begin{align*}
    \langle\partial_t\varphi^n(t),\psi \rangle_{V_s} + \langle P^n\mathcal{C}_1({\bu}^n(t),\varphi^n(t)),\psi \rangle_{V_s} + \langle P^nB(\mu^n(t)),\psi \rangle_{V_s} = 0,
\end{align*}
with $\mu^n = \rho(\cdot,\varphi^n) - J\ast \varphi^n + \ell_c\theta^n$. Before going further, let us mention that the estimates will highly depend on $p\in(1,2]$ from assumption (A5) due to the appearance of $\rho(\cdot,\varphi^n)$ in the definition of $\mu^n$ which we have proven to be bounded in $L^\infty(I;L^p(\Omega))$. 
Due to the limited of regularity of $\rho(\cdot,\varphi^n)$ we are compelled to consider test functions $\psi\in V_s$ for some $s\ge2$ that will allow us to utilize the Sobolev embedding $H^{s-2}(\Omega)\hookrightarrow L^{p'}(\Omega)$ where $p'$ is the H{\"o}lder conjugate of $p$. If such embedding holds then H{\"o}lder's inequality, (A1), and \eqref{estimate:rho} will imply 
\begin{align*}
    \begin{aligned}
        |\langle P^nB(\mu^n(t)),\psi \rangle_{V_s}|
        &\le |(\nabla \rho(\cdot,\varphi^n),\nabla\psi)_\Omega| + |(\nabla J\ast\varphi^n,\nabla\psi)_\Omega| + \ell_c|(\nabla \theta^n,\nabla\psi)_\Omega|\\
        &\le |(\rho(\cdot,\varphi^n),\Delta\psi)_\Omega| + \|\nabla J \|_{L^1}\|\varphi^n\|\|\nabla\psi\| + \ell_c|( \theta^n,\Delta\psi)_\Omega|\\
        &\le \|\rho(\cdot,\varphi^n)\|_{L^p}\|\Delta\psi\|_{L^{p'}} + \|\nabla J \|_{L^1}\|\varphi^n\|\|\psi\|_{V_s} + \ell_c\|\theta^n\|\|\Delta\psi\|_{L^{p'}}\\
        &\le (\|\rho(\cdot,\varphi^n)\|_{L^p}+ \|\nabla J \|_{L^1}\|\varphi^n\| + \ell_c\|\theta^n\|)\|\psi\|_{V_s}.
    \end{aligned}
\end{align*}
Using \eqref{uniformbounds:phiH}, \eqref{uniformbounds:thetaHV} and \eqref{estimate:rho}, we thus get
\begin{align*}
    \|P^nB(\mu^n(t))\|_{L^\infty(I;V_s^*)} \le M.
\end{align*}

Let us now take our gaze onto the term $\langle P^n\mathcal{C}_1({\bu}^n(t),\varphi^n(t)),\psi \rangle_{V_s}$.
We know that the embedding $H^{s-2}(\Omega)\hookrightarrow L^{p'}(\Omega)$ --- to maintain the validity of the previous computation --- holds if either of the following cases is true:
\begin{itemize}
    \item for $2\le p' \le +\infty$ if $(s-2)2>d$,
    \item for $2\le p' < +\infty$ if $(s-2)2 = d$,
    \item for $2\le p' \le 2d/(d-(s-2)2)$ if $(s-2)2<d$.
\end{itemize}

The first two cases instantly yields a good estimate for $\langle P^n\mathcal{C}_1({\bu}^n(t),\varphi^n(t)),\psi \rangle_V$. Indeed, from H{\"o}lder's inequality, and from the fact that $(s-2)2 \ge d$ implies $(s-1)2>d$ --- which validates the embedding $H^{s-1}(\Omega)\hookrightarrow L^\infty(\Omega)$ --- we get
\begin{align}
    \begin{aligned}
        |\langle P^n\mathcal{C}_1({\bu}^n(t),\varphi^n(t)),\psi \rangle_{V_s}| & = |({\bu}^n\cdot\nabla \psi,\varphi^n )_\Omega| \le \|{\bu}^n\|\|\nabla\psi\|_{L^\infty}\|\varphi^n\|\\
        &\le \|{\bu}^n\|\|\nabla\psi\|_{H^{s-1}}\|\varphi^n\|\le \|{\bu}^n\|\|\varphi^n\|\|\psi\|_{V_s}
    \end{aligned}\label{estimate:pc1i}
\end{align}

The third case and the fact that $p'$ is the H{\"o}lder conjugate of $p$ implies that it would suffice to just consider $\frac{(4-d)p + 2d}{2p} \le s < \frac{d+4}{2}$.
The arbitrariness of $p\in(1,2]$ also calls for us to consider $p\in (1,d/(d-1))$, $p=d/(d-1)$, and $p\in (3/2,2]$ exclusively for $d =3$. Let us first consider the following cases:
\begin{itemize}
    \item $\frac{(4-d)p + 2d}{2p} \le s$ and $p\in (1,d/(d-1))$,
    \item $\frac{(4-d)p + 2d}{2p} < s$ and $p = d/(d-1)$.
\end{itemize}
We see that the scenarios above imply $(s-1)2>d$, hence the embedding $H^{s-1}(\Omega)\hookrightarrow L^\infty(\Omega)$. Following the same computations as in \eqref{estimate:pc1i}, we get
\begin{align*}
    \begin{aligned}
        |\langle P^n\mathcal{C}_1({\bu}^n(t),\varphi^n(t)),\psi \rangle_{V_s}| & \le \|{\bu}^n\|\|\varphi^n\|\|\psi\|_{V_s}
    \end{aligned}
\end{align*}
Using \eqref{uniformbounds:phiH} and \eqref{uniformbounds:uHV} we see that
\begin{align*}
    \|P^n\mathcal{C}_1({\bu}^n(t),\varphi^n(t))\|_{L^\infty(I;V_s^*)} \le M^2,
\end{align*}
if either $(s-2)2\ge d$, $\frac{(4-d)p + 2d}{2p} \le s$  and  $p\in (1,d/(d-1))$, or  $\frac{(4-d)p + 2d}{2p} < s$ and  $p = d/(d-1)$.

The case $\frac{(4-d)p + 2d}{2p} = s$ and $p = d/(d-1)$ implies $(s-1)2=d$ which gives us the embedding $H^{s-1}(\Omega)\hookrightarrow L^r(\Omega)$ for any $2< r <+\infty$. Using H{\"o}lder's inequality, the embedding previously mentioned, and the interpolation in $L^p$ with $\theta = (r-4)/r$, $p = 2$, and $q = 4$, we arrive at the following computation:
\begin{align*}
    \begin{aligned}
        |\langle P^n\mathcal{C}_1({\bu}^n(t),\varphi^n(t)),\psi \rangle_{V_s}| & \le \|{\bu}^n\|\|\nabla\psi\|_{L^r}\|\varphi^n\|_{L^{\frac{2r}{r-2}}}\\ 
         & \le \|{\bu}^n\|\|\nabla\psi\|_{H^{s-1}}\|\varphi^n\|^{\frac{r-4}{r}}\|\varphi^n\|_{L^4}^{\frac{4}{r}}\\
        & \le \|{\bu}^n\|\|\psi\|_{V_s}\|\varphi^n\|^{\frac{r-4}{r}}\|\varphi^n\|_{V}^{\frac{4}{r}}.
    \end{aligned}
\end{align*}
Using \eqref{uniformbounds:phiH}, \eqref{uniformbounds:uHV} and \eqref{estimate:phiVV}, we get that the case currently being considered yields 
\begin{align*}
    \begin{aligned}
        \|P^n\mathcal{C}_1({\bu}^n,\varphi^n)\|_{L^{\frac{r}{2}}(I;V_s^*)} & \le \left(\int_0^T (\|{\bu}^n(t)\|\|\varphi^n(t)\|^{\frac{r-4}{r}}\|\varphi^n(t)\|_{V}^{\frac{4}{r}})^{\frac{r}{2}} \du t \right)^{2/r}\\
        & \le \|{\bu}^n\|_{L^\infty(I;H_\sigma)}\|\varphi^n\|_{L^\infty(I;H)}^{\frac{r-4}{r}}\|\varphi^n\|_{L^2(I;V)}^{\frac{4}{r}}
    \end{aligned}
\end{align*}

Lastly, $\frac{3}{2}<p\le 2$ and $s = \frac{(4-d)p + 2d}{2p} = \frac{1}{2} + \frac{3}{p}$ we get that $(s-1)2 < 3 = d$. This implies the embedding $H^{s-1}(\Omega)\hookrightarrow L^r(\Omega)$ with $r = \frac{3p}{2p-3}$. By letting $r'>1$ be such that $\frac{1}{r} + \frac{1}{r'} = \frac{1}{2}$, together with $L^p$ interpolation (with $\theta = \frac{20p-30}{9p}$, $p_1=5$, and $p_2=2$), and the embedding $H^1(\Omega)\hookrightarrow L^5(\Omega)$  we get the following estimate
\begin{align*}
    \begin{aligned}
        |\langle P^n\mathcal{C}_1({\bu}^n(t),\varphi^n(t)),\psi \rangle_{V_s}| & \le \|{\bu}^n\|\|\nabla\psi\|_{L^r}\|\varphi^n\|_{L^{r'}}\\
        & \le \|{\bu}^n\|\|\nabla\psi\|_{H^{s-1}}\|\varphi^n\|^{\frac{30-11p}{9p}}\|\varphi^n\|_{L^5}^{\frac{10}{9}\frac{2p-3}{p}}\\
        & \le \|{\bu}^n\|\|\psi\|_{V_s}\|\varphi^n\|^{\frac{30-11p}{9p}}\|\varphi^n\|_{V}^{\frac{10}{9}\frac{2p-3}{p}}.
    \end{aligned}
\end{align*} 
From the $L^\infty(I;H_\sigma)$ and $L^\infty(I;H)$ estimates of ${\bu}^n$ and $\varphi^n$, respectively, and the $L^2(I;V)$ estimate of $\varphi^n$ from \eqref{estimate:phiVV} we infer that
\begin{align*}
    \begin{aligned}
        \|P^n\mathcal{C}_1({\bu}^n,\varphi^n)\|_{L^{\frac{9p}{5(2p-3)}}(I;V_s^*)} & \le \left(\int_0^T \left(\|{\bu}^n(t)\|\|\varphi^n\|^{\frac{30-11p}{9p}}\|\varphi^n(t)\|_{V}^{\frac{10}{9}\frac{2p-3}{p}} \right)^{\frac{9p}{5(2p-3)}} \du t \right)^{\frac{5(2p-3)}{9p}}\\
        & \le \|{\bu}^n\|_{L^\infty(I;H_\sigma)}\|\varphi^n\|^{\frac{30-11p}{9p}}_{L^\infty(I;H)}\|\varphi^n\|_{L^2(I;V)}^{\frac{10}{9}\frac{2p-3}{p}}.
    \end{aligned}
\end{align*}
From all these we finally get the following estimates for $\partial_t\varphi^n$.
\begin{align}
    \begin{aligned}
        \|\partial_t\varphi^n\|_{ L^\infty(I;V_s^*)\cap L^{\frac{r}{2}}(I;V_{\frac{d+2}{2}}^*) } \le M+M^2,
        \label{estimate:timephi1}
    \end{aligned}
\end{align}
where $r>2$ and if any of the following cases hold:
\begin{itemize}
    \item $(s-2)2 \ge d$,
    \item $1<p<\frac{d}{d-1}$ with $\frac{d+4}{2} > s\ge \frac{(4-d)p+2d}{2p}$, or 
    \item $p = \frac{d}{d-1}$ with $\frac{d+4}{2} > s>\frac{(4-d)p+2d}{2p}$.
\end{itemize}
Lastly, if $d = 3$ and $\frac{3}{2}<p\le 2$ we have
\begin{align}
    \begin{aligned}
        \|\partial_t\varphi^n\|_{ L^{\frac{9p}{5(2p-3)}}(I;V_s^*) } \le M + M^2.
    \end{aligned}\label{estimate:timephi2}
\end{align}

For the time derivative of the temperature, we rewrite \eqref{weakdisc:theta} as 
\begin{align}
    \begin{aligned}
        &\langle\partial_tH(\theta^n(t),\varphi^n(t)),\vartheta \rangle_{V_s}+ \langle P^n\mathcal{C}_1({\bu}^n(t),H(\theta^n(t),\varphi^n(t))),\vartheta\rangle_{V_s}\\ & + \kappa\langle P^n\mathcal{B}\theta^n(t),\vartheta\rangle_{V_s} = \langle P^n({\bg}\cdot{\bu}^n(t)),\vartheta\rangle_{V_s} + \langle z^n(t),\vartheta\rangle_{V_s},
\end{aligned} \label{projectedeq:thetan}
\end{align}
where $H(\theta^n(t),\varphi^n(t))=\theta^n(t)-\ell_h\varphi^n(t)$. Apparently, we also have to work on the space $V_s$ due to the appearance of the time derivative of the order parameter. Of course, the linear parts of \eqref{projectedeq:thetan} can be handled conveniently, and in fact --- because of the embedding $H^{s-2}(\Omega)\hookrightarrow L^2(\Omega)$ --- we have
\begin{align*}
    \|P^n\mathcal{B}\theta^n(t)\|_{V_s^*} & = \inf_{\substack{{\vartheta}\in V_s\\ \|\vartheta\|_{V_s} = 1}} |\langle P^n\mathcal{B}\theta^n(t),\vartheta\rangle_{V_s}| = \inf_{\substack{{\vartheta}\in V_s\\ \|\vartheta\|_{V_s} = 1}} |( \theta^n(t),\Delta\vartheta)_\Omega|\\
    & = \inf_{\substack{{\vartheta}\in V_s\\ \|\vartheta\|_{V_s} = 1}} \| \theta^n(t)\|\|\Delta\vartheta\| \le  \| \theta^n(t)\|.
\end{align*}
Hence, we obtain from \eqref{uniformbounds:thetaHV}
\begin{align*}
    \|P^n\mathcal{B}\theta^n\|_{L^\infty(I;V_s^*)} \le \| \theta^n\|_{L^\infty(I;H)} \le M.
\end{align*}
H{\"o}lder's inequality, the constancy of the gravitational parameter $\bg$ and \eqref{uniformbounds:uHV}, on the other hand, give us
\begin{align*}
    \|P^n({\bg}\cdot{\bu}^n)\|_{L^\infty(I;V_s^*)} \le c\| {\bu}^n\|_{L^\infty(I;H_\sigma)} \le M.
\end{align*}

For the transport term, we shall take advantage of the embedding $H^{s-1}(\Omega)\hookrightarrow L^r(\Omega)$ for some $2<r\le +\infty$ as illustrated in the previous computations for the time derivative of $\varphi^n$. For this purpose, we shall skip some parts which we have covered previously, but mention the values of $r$ that we used. 

\noindent\textbf{Case 1: $r = +\infty$}. Before we begin, let us mention that due to \eqref{uniformbounds:phiH}, \eqref{uniformbounds:thetaHV} and \eqref{estimate:phiVV} we get that $H(\theta^n,\varphi^n)$ belongs to $L^\infty(I;H)\cap L^2(I;V)$, and in fact we have
\begin{align}
\begin{aligned}
    \| H(\theta^n,\varphi^n)\|_{L^\infty(I;H)\cap L^2(I;V)} \le M.
\end{aligned} \label{estimate:Hthph}
\end{align}
Such property is global and does not only cover the current case. 

Going back, we see that by utilizing H{\"o}lder's inequality, we get
\begin{align*}
\begin{aligned}
    \| P^n\mathcal{C}_1({\bu}^n(t), H(\theta^n(t),\varphi^n(t))) \|_{V_s^*} & = \inf_{\substack{{\vartheta}\in V_s\\ \|\vartheta\|_{V_s} = 1}} |\langle P^n\mathcal{C}_1({\bu}^n(t),H(\theta^n(t),\varphi^n(t))),\vartheta\rangle_{V_s}|\\
    & = \inf_{\substack{{\vartheta}\in V_s\\ \|\vartheta\|_{V_s} = 1}} |( {\bu}^n(t)\cdot\nabla \vartheta, H(\theta^n(t),\varphi^n(t)) )_\Omega|\\
    & \le \inf_{\substack{{\vartheta}\in V_s\\ \|\vartheta\|_{V_s} = 1}} \|{\bu}^n(t)\|\|\nabla\vartheta\|_{L^\infty}\|H(\theta^n(t),\varphi^n(t))\|\\
    & \le \|{\bu}^n(t)\|\|H(\theta^n(t),\varphi^n(t))\|.
\end{aligned}
\end{align*}
From \eqref{uniformbounds:phiH} and \eqref{uniformbounds:uHV} we hence achieve
\begin{align*}
    \| P^n\mathcal{C}_1({\bu}^n(t), H(\theta^n(t),\varphi^n(t))) \|_{L^\infty(I;V_s^*)} \le M^2.
\end{align*}

\noindent\textbf{Case 2: $2< r < +\infty$}. This particular case is divided in two subcases. The first one is when $\frac{(4-d)p + 2d}{2p} = s$ and $p = d/(d-1)$, which implies $2< r < +\infty$ can be chosen arbitrarily. In this instance, we see that
\begin{align*}
\begin{aligned}
    \| P^n\mathcal{C}_1({\bu}^n(t), H(\theta^n(t),\varphi^n(t))) \|_{V_s^*}
    & = \inf_{\substack{{\vartheta}\in V_s\\ \|\vartheta\|_{V_s} = 1}} |\langle P^n\mathcal{C}_1({\bu}^n(t),H(\theta^n(t),\varphi^n(t))),\vartheta\rangle_{V_s}|\\
    & = \inf_{\substack{{\vartheta}\in V_s\\ \|\vartheta\|_{V_s} = 1}} |( {\bu}^n(t)\cdot\nabla \vartheta, H(\theta^n(t),\varphi^n(t)) )_\Omega|\\
    & \le \inf_{\substack{{\vartheta}\in V_s\\ \|\vartheta\|_{V_s} = 1}} \|{\bu}^n(t)\|\|\nabla\vartheta\|_{L^r}\|H(\theta^n(t),\varphi^n(t))\|_{L^{\frac{2r}{r-2}}}\\
    & \le  \|{\bu}^n(t)\|\|H(\theta^n(t),\varphi^n(t))\|^{\frac{r-4}{r}}\|H(\theta^n(t),\varphi^n(t))\|_{V}^{\frac{4}{r}}.
\end{aligned}
\end{align*}
Here, the last inequality is achieved by additionally employing $L^p$-interpolation with $\theta = (r-4)/r$, $p_1=4$ and $p_2=2$, and the embedding $H^1(\Omega)\hookrightarrow L^4(\Omega)$. We thus get 
\begin{align*}
    \begin{aligned}
        &\| P^n\mathcal{C}_1({\bu}^n, H(\theta^n,\varphi^n)) \|_{L^{\frac{r}{2}}(I;V_s^*)} \\& \le \left( \int_0^T \left( \|{\bu}^n(t)\| \|H(\theta^n(t),\varphi^n(t))\|^{\frac{r-4}{r}} \|H(\theta^n(t),\varphi^n(t))\|_{V}^{\frac{4}{r}} \right)^{\frac{r}{2}} \du t \right)^{2/r}\\
         & \le \|{\bu}^n\|_{L^\infty(I;H_\sigma)} \|H(\theta^n,\varphi^n)\|_{L^\infty(I;H)}^{\frac{r-4}{r}} \|H(\theta^n(t),\varphi^n(t))\|_{L^2(I;V)}^{\frac{4}{r}} \le M^2,
    \end{aligned}
\end{align*}
where the last inequality follows from \eqref{uniformbounds:uHV} and \eqref{estimate:Hthph}.

The second subcase is when $d=3$ and $\frac{3}{2}< p\le 2$, which gives us a particular value $r = \frac{3p}{2p-3}$. By additionally utilizing $L^p$ interpolation with $\theta = \frac{30-11p}{9p}$, $p_1=5$, and $p_2=2$ and the embedding $H^1(\Omega)\hookrightarrow L^5(\Omega)$ we get
\begin{align*}
\begin{aligned}
    \| P^n\mathcal{C}_1({\bu}^n(t),   H(\theta^n(t),\varphi^n(t))) \|_{V_s^*} 
    & = \inf_{\substack{{\vartheta}\in V_s\\ \|\vartheta\|_{V_s} = 1}} |\langle P^n\mathcal{C}_1({\bu}^n(t),H(\theta^n(t),\varphi^n(t))),\vartheta\rangle_{V_s}|\\
    & = \inf_{\substack{{\vartheta}\in V_s\\ \|\vartheta\|_{V_s} = 1}} |( {\bu}^n(t)\cdot\nabla \vartheta, H(\theta^n(t),\varphi^n(t)) )_\Omega|\\
    & \le \inf_{\substack{{\vartheta}\in V_s\\ \|\vartheta\|_{V_s} = 1}} \|{\bu}^n(t)\|\|\nabla\vartheta\|_{L^r}\|H(\theta^n(t),\varphi^n(t))\|_{L^{\frac{2r}{r-2}}}\\
    & \le  \|{\bu}^n(t)\|\|H(\theta^n(t),\varphi^n(t))\|^{\frac{30-11p}{9p}} \|H(\theta^n(t),\varphi^n(t))\|_{V}^{\frac{10}{9}\frac{2p-3}{p}} 
\end{aligned}
\end{align*}
We further estimate with respect to the time variable as follows:
\begin{align*}
\begin{aligned}
   & \| P^n\mathcal{C}_1({\bu}^n, H(\theta^n,\varphi^n)) \|_{L^{\tilde{p}}(I;V_s^*)}\\ & \le \left(\int_0^T \left( \|{\bu}^n(t)\| \|H(\theta^n(t),\varphi^n(t))\|^{\frac{30-11p}{9p}} \|H(\theta^n(t),\varphi^n(t))\|_{V}^{\frac{2}{\tilde{p}}}\right)^{{\tilde{p}}}\du t \right)^{\frac{1}{\tilde{p}}}\\
   & \le \|{\bu}^n\|_{L^\infty(I;H_\sigma)} \|H(\theta^n,\varphi^n)\|_{L^\infty(I;H_\sigma)}^{\frac{30-11p}{9p}} \|H(\theta^n,\varphi^n)\|_{L^2(I;V)}^{\frac{10}{9}\frac{2p-3}{p}} \le M^2,
\end{aligned}
\end{align*}
where $\tilde{p} = \frac{9p}{5(2p-3)}$.

Finally, the computations above give us the estimate for the temperature variable $\partial_t\theta^n$
\begin{align}
    \begin{aligned}
        \|\partial_t\theta^n\|_{ L^\infty(I;V_s^*)\cap L^{\frac{r}{2}}(I;V_{\frac{d+2}{2}}^*) } \le M+M^2
    \end{aligned}\label{estimate:timetheta1}
\end{align}
where $r>2$ and if any of the following cases hold:
\begin{itemize}
    \item $(s-2)2 \ge d$,
    \item $1<p<\frac{d}{d-1}$ with $\frac{d+4}{2} > s\ge \frac{(4-d)p+2d}{2p}$, or 
    \item $p = \frac{d}{d-1}$ with $\frac{d+4}{2} > s>\frac{(4-d)p+2d}{2p}$.
\end{itemize}
And if $d = 3$ and $\frac{3}{2}<p\le 2$ we have
\begin{align}
    \begin{aligned}
        \|\partial_t\theta^n\|_{ L^{\frac{9p}{5(2p-3)}}(I;V_s^*) } \le M + M^2.
    \end{aligned}\label{estimate:timetheta2}
\end{align}

The estimates \eqref{uniformbounds:phiH}--\eqref{uniformbounds:thetaHV}, \eqref{estimate:phiVV}, \eqref{boundedness:muV} and \eqref{estimate:rho} imply the existence of $\varphi\in L^\infty(I;H)\cap L^2(I;V)$, $\mu \in L^2(I;V)$, $\rho\in L^\infty(I;L^p(\Omega)) $, $\bu \in L^\infty(I;H_\sigma)\cap L^2(I;V_\sigma)$ and $\theta\in L^\infty(I;H)\cap L^2(I;V)$ for which --- up to a sebsequence --- the following properties hold
\begin{align}
    &\varphi^n \rightharpoonup \varphi &&\text{in }L^2(I;V)\label{conv:phiwl2}\\
    &\varphi^n \ws \varphi &&\text{in }L^\infty(I;H)\label{conv:vphi}\\
    &\mu^n \rightharpoonup \mu &&\text{in }L^2(I;V)\label{conv:mu}\\
    &\rho(\cdot,\varphi^n) \ws \rho &&\text{in }L^\infty(I;L^p(\Omega))\label{con:vrho}\\
    &\bu^n \rightharpoonup \bu &&\text{in }L^2(I;V_\sigma)\label{conv:ul2v}\\
    &\bu^n \ws \bu &&\text{in }L^\infty(I;H_\sigma)\\
    &\theta^n \rightharpoonup \theta &&\text{in }L^2(I;V)\\
    &\theta^n \ws \theta &&\text{in }L^\infty(I;H).\label{conv:thetaH}
\end{align}
Estimates \eqref{estimate:timeu}, \eqref{estimate:timephi1}, \eqref{estimate:timephi2}, \eqref{estimate:timetheta1} and \eqref{estimate:timetheta2}, further gives us the convergences of the time derivatives:
\begin{align}
    & \partial_t{\bu}^n \rightharpoonup \partial_t{\bu} \text{ in }L^{4/d}(I;V_\sigma^*),\quad \partial_t \varphi^n \rightharpoonup \partial_t\varphi \text{ and }  \partial_t \theta^n \rightharpoonup \partial_t\theta \text{ in }L^q(I;V^*_s)
\end{align}
for $q = \frac{9p}{5(2p-3)}$ and $s = \frac{p+6}{2p}$ if $d = 3$ and $\frac{3}{2}<p\le 2$, and for $q = \frac{r}{2}$ and $s = \frac{d+2}{2}$ --- with $r>2$ taken arbitrarily --- if $p = \frac{d}{d-1}$ and $d = 2,3$,
\begin{align}
    & \partial_t \varphi^n \ws \partial_t\varphi,\qquad  \partial_t \theta^n \ws \partial_t\theta &&\text{in }L^\infty(I;V^*_s)
\end{align}
for either $(s-2)2\ge d$, $1<p<\frac{d}{d-1}$ with $\frac{d+4}{2}>s\ge\frac{(4-d)p+2d}{2p}$, or $p=\frac{d}{d-1}$ with $\frac{d+4}{2}>s>\frac{(4-d)p+2d}{2p}$.

Taking advantage of Aubin--Lions--Simon embedding theorems would give us the compact embedding $W^{2,4/d}(V_\sigma,V_\sigma^*)\hookrightarrow L^2(I;H_\sigma)$. Similarly, the embedding $W^{2,q}(V,V_s^*)\hookrightarrow L^2(I;H)$ holds for the following cases
\begin{itemize}
    \item $q = \frac{9p}{5(2p-3)}$ and $s = \frac{p+6}{2p}$ if $d = 3$ and $\frac{3}{2}<p\le 2$;
    \item $q = \frac{r}{2}$, $r>2$ taken arbitrarily, and $s = \frac{d+2}{2}$ if $p = \frac{d}{d-1}$ $d = 2,3$; and
    \item $q = +\infty$ if either $(s-2)2\ge d$; $1<p<\frac{d}{d-1}$ with $\frac{d+4}{2}>s\ge\frac{(4-d)p+2d}{2p}$; or $p=\frac{d}{d-1}$ with $\frac{d+4}{2}>s>\frac{(4-d)p+2d}{2p}$.
\end{itemize}
From these, we get the following strong convergences
\begin{align}
    &\bu^n \to \bu \text{ in }L^2(I;H_\sigma),\quad \varphi^n \to \varphi \text{ in }L^2(I;H),\quad \theta^n \to \theta \text{ in }L^2(I;H)
    \label{strongconvergence}
\end{align}
all of which converge a.e. in $\Omega\times(0,T)$.

The passage of the limit for the order parameter and the fluid velocity follows that of in \cite{colli2012}. Although such steps would be straightforward to establish, we highlight the crucial points for which we prove that the triple $[\varphi,{\bu},\theta]$ is a weak solution of the system in the sense of Definition \ref{definition:weak}. Firstly, we point out some direct consequences of \eqref{strongconvergence}. 
\begin{itemize}
    \item By definition of $\rho(\cdot,\varphi^n)$, the regularity of $F$, and the pointwise convergence of the order parameter, $\rho(\cdot,\varphi^n)\to a(\cdot)\varphi + F'(\varphi)$ a.e. in $\Omega\times(0,T)$, from which we also infer $\rho = a(\cdot)\varphi + F'(\varphi)$.
    \item The strong convergence of the order parameter in $L^2(I;H)$ implies $J\ast\varphi^n\to J\ast\varphi$ in $L^2(I;V)$. Such convergence and the definition of $\mu^n$ imply $\mu = \rho - J\ast\varphi+ \ell_c\theta$. Indeed, for any $\psi\in H^n$ and $\chi\in C_0^\infty(0,T)$, we get
    \begin{align*}
        \int_0^T (\mu^n, \psi)_\Omega\chi\du t & = \int_0^T (\rho(\cdot,\varphi^n)-J\ast\varphi^n + \ell_c\theta^n, \psi)_\Omega\chi\du t\\& \to \int_0^T (\rho-J\ast\varphi + \ell_c\theta, \psi)_\Omega\chi\du t,
    \end{align*}
     where we used the pointwise convergence of $\rho$, the strong convergence of convolution, the strong convergence of the temperature variable in $L^2(I;H)$, and the density of $\rm{span}\{\psi_j\}$ in $V$. We thus infer that $\mu = \rho-J\ast\varphi + \ell_c\theta$ from \eqref{conv:mu}, and that $\rho \in L^2(I;V)$.
\end{itemize}
Moving forward, we multiply \eqref{weakdisc:phi}, \eqref{weakdisc:u}, \eqref{weakdisc:theta} by $\chi,\mathfrak{w},\omega \in C_0^\infty(0,T)$, and integrate over the interval $(0,T)$ and take advantage of \eqref{conv:phiwl2}--\eqref{strongconvergence}. Since some of the terms in the resulting integral equation can be handled quite easily, we highlight only some of the parts which we deemed to be crucial. 
\begin{itemize}
    \item The term $(\nabla\rho(\cdot,\varphi^n),\nabla\psi)_\Omega$ is handled by passing the derivative to the test function, i.e. $-(\rho(\cdot,\varphi^n),\Delta\psi)$, and utilizing \eqref{con:vrho} with the knowledge that $\psi\in V_s$ and $H^{s-2}\hookrightarrow L^{p'}$.
    \item Due to the convergence $J\ast\varphi^n\to J\ast\varphi$ in $L^2(I;V)$ we also achieve convergence for the term $(\nabla J\ast\varphi^n,\nabla\psi)_{\Omega}$.
    \item The expression involving the transport term $({\bu}^n(t)\cdot\nabla\varphi^n(t),\psi)_\Omega$ can be established to converge to $({\bu}(t)\cdot\nabla\varphi(t),\psi)_\Omega$ by utilizing the convergences \eqref{conv:ul2v}, \eqref{conv:phiwl2} and \eqref{strongconvergence}. The same argument is used to show the convergence of the transport term occurring in the equation for the temperature, i.e., the term $({\bu}^n(t)\cdot\nabla(\theta^n(t)-\ell_h\varphi^n(t)),\vartheta)_\Omega$.
    \item The strong convergence of the order parameter in \eqref{strongconvergence}, the imposed properties of $\nu(\cdot)$ in (A2), and the dominated convergence theorem imply $\nu(\varphi^n)\to \nu(\varphi)$ in $L^p(I;L^p(\Omega))$ for any $1\le p <+\infty$. From the convergence of the viscosity term above and \eqref{conv:ul2v} we implore that $\nu(\varphi^n)\mathrm{D}{\bu}^n \rightharpoonup \nu(\varphi)\mathrm{D}{\bu}$ in $L^2(I;L^2(\Omega)^{d\times d})$. The weak and strong convergences of the velocity vector in $L^2(I;V_\sigma)$ and $L^2(I;H_\sigma)$, respectively, give us the convergence for the trilinear form in \eqref{weakdisc:u} \cite[Lemma III.3.2]{temam1977}.
\end{itemize}
The linear terms can be handled quite easily using the convergences we have in hand. For the terms concerning the time derivative, we pass the time derivatives to the functions $\chi,\mathfrak{w},\omega \in C_0^\infty(0,T)$ and use \eqref{strongconvergence}. From these we see that the triple $[\varphi,{\bu},\theta]$ satisfies \eqref{weak:phi}, \eqref{weak:u} and \eqref{weak:theta} by virtue of the density of $\rm{span}\{{\bv}_j\}$ and $\rm{span}\{\psi_j\}$ in $V_\sigma$ and $V$.

Now we suppose $\chi,\mathfrak{w},\omega \in C_\infty(0,T)$ with values equal to one at the initial time $t= 0$ and equal to zero at the terminal time $t=T$ and apply the same passage of the limit. Multiplying the same functions to \eqref{weak:phi}, \eqref{weak:u} and \eqref{weak:theta} and apply integration by parts eventually lead to 
    \begin{align*}
			&({\varphi}(0)-{\varphi}_0,{\psi})_{\Omega} = 0 \quad\forall {\psi}\in V,\\
			&({\bu}(0)-{\bu}_0,{\bv})_{\Omega} = 0 \quad\forall {\bv}\in V_\sigma,\\
			&({\theta}(0)-{\theta}_0,{\vartheta})_{\Omega} = 0 \quad\forall {\vartheta}\in V.
    \end{align*}
We also mention that due to the Aubin--Lions lemma we have the embeddings $W^{+\infty,4/d}(H_\sigma,V_\sigma^*)\hookrightarrow C(\overline{I};H_\sigma)$ and $W^{+\infty,q}(H,V_s^*)\hookrightarrow C(\overline{I};H)$. These validate the evaluation of the weak solutions at the initial time $t=0$, and shows the satisfaction of \eqref{initcon:phi}, \eqref{initcon:u} and \eqref{initcon:theta}. 

The improved regularity of $\rho \in L^2(I;V)$ lets us write \eqref{weak:phi} as
\begin{align*}
    \langle \partial_t\varphi(t),\psi \rangle_V + \langle \mathcal{C}_1({\bu}(t),\varphi(t)),\psi\rangle_V + \langle B\mu(t),\psi \rangle_V & = 0 &&\forall \psi\in V.
\end{align*}
The transport term can thus be estimated --- following analogous arguments as in \eqref{estimate:capillarityD3} and \eqref{estimate:capillarityD2} --- as
\begin{align*}
    \|\mathcal{C}_1({\bu},\varphi)\|_{L^{4/d}(I;V^*)} & \le M^2.
\end{align*}
This improves the regularity of the time derivative as well, i.e. $\partial_t\varphi\in L^{4/d}(I,V^*)$. The space upon which the time derivative of the temperature inherits the improvement gained by the order parameter as well, i.e. $\partial_t\theta\in L^{4/d}(I,V^*)$.

Another consequence of the embeddings above is that up to a subsequence we have $\varphi^n(t)\to \varphi(t)$ in $H$ and almost everywhere in $\Omega$. Consequently, by virtue of Fatou's lemma
\begin{align}
    \int_\Omega F(\varphi(t)) \du x \le \liminf_{n\to\infty} \int_\Omega F(\varphi^n(t)) \du x.\label{liminf:F}
\end{align}
We also have, from the weak lower semi-continuity of the nonlocal energy $\mathbb{E}_{nl}$ and the weak convergence of $\varphi^n$ to $\varphi$, 
\begin{align}
    \int_\Omega \mathbb{E}_{nl}(\varphi(t)) \du x \le \liminf_{n\to\infty} \int_\Omega \mathbb{E}_{nl}(\varphi^n(t)) \du x.\label{liminf:Enl}
\end{align}
We recall that $J\ast\varphi^n\to J\ast\varphi$ in $L^2(I;V)$ which implies $P^n(J\ast\varphi^n) \to J\ast\varphi$ in $L^2(I;V)$, and $\sqrt{\nu(\varphi^n)}\mathrm{D}{\bu}^n \rightharpoonup \sqrt{\nu(\varphi)}\mathrm{D}{\bu}$ in $L^2(I;L^2(\Omega)^{d\times d})$. Integrating \eqref{eqreconcile} over $[0,t]$, utilizing \eqref{liminf:F}, \eqref{liminf:Enl} and \eqref{conv:phiwl2}--\eqref{conv:thetaH}, and taking advantage of the weak lower semi-continuity of norms, we achieve the energy inequality \eqref{energyineq}.

\end{proof}

\section{Nonlocal-to-Local Convergence}

The purpose of the section is to show, under an appropriate choice of the kernel $J$, that solutions of the nonlocal system converges to its local version. Notably, we should be able to determine by which convergence we shall anchor on, we also recall that the only assumption we imposed for the kernel is (A1), so it is imperative that such assumption is not violated by the particular kinds of kernel we shall consider.

Suppose $\gamma\in (0,d-1)$, $\eta\in C^1([0,+\infty);[0,+\infty))$ such that the map $s\mapsto |\eta'(s)|s^{d-1-\gamma}$ is in $L^1(\mathbb{R}^+)$, and satisfies the renormalization 
\begin{align*}
    \int_0^{+\infty} \eta(s)s^{d+1-\gamma} \du s = \frac{2}{C_d}, \text{ where }C_d:= \int_{S^{d-1}} |\sigma\cdot e_1|^{2}\du \mathcal{H}^{d-1}(\sigma).
\end{align*}
We define the family of mollifiers $(\eta_\varepsilon)_{\varepsilon>0}$ as $\eta_\varepsilon(s) = \frac{1}{\varepsilon^d}\eta(s/\varepsilon)$ for $s\ge 0$, from which we define the kernel $J_\varepsilon:\mathbb{R}^d\to\mathbb{R}$ by $J_\varepsilon(x) = \frac{1}{\varepsilon^{2-\gamma}}\frac{\eta_\varepsilon(|x|)}{|x|^\gamma}$. Evidently, for any $\varepsilon>0$ $J_\varepsilon(x) = J_\varepsilon(-x)$ for all $x\in\Omega$ and $a_\varepsilon(x) = \int_\Omega J_\varepsilon(x-y)\du y \ge 0$ for almost every $x\in\Omega$. Meanwhile, from \cite[Lemma 3.1]{DAVOLI2021a} we infer that $J_\varepsilon\in W^{1,1}(\mathbb{R}^d)$.

For a given $\varepsilon>0$ we now define the corresponding nonlocal energy functional to each kernel $J_\varepsilon$ as
\begin{align*}
    \mathbb{E}_{nl}^\varepsilon(\varphi) := \frac{1}{2} \int_\Omega\int_\Omega J_\varepsilon(x-y)(\varphi(x)-\varphi(y))^2 \du y\du x.
\end{align*}
And for each $\epsilon>0$ and each kernel $J_\varepsilon$, Theorem \ref{th:existence} assures us of the existence of solutions of \eqref{system:nlCHOB} in the sense of Definition \ref{definition:weak} which we shall denote as $[\varphi_\varepsilon,{\bu}_\varepsilon,\theta_\varepsilon]$ together with the chemical potential $\mu_\varepsilon:= a_\varepsilon\varphi_\varepsilon - J_\varepsilon\ast\varphi_\varepsilon + F'(\varphi_\varepsilon) + \ell_c\theta_\varepsilon$. We also denote by $\mathbb{E}^\varepsilon$ the total energy of the system \eqref{system:nlCHOB}  which we recall to be written as in \eqref{totalenergynl} but now with the kernel $J_\varepsilon$.

Our goal is to establish that the solutions $[\varphi_\varepsilon,{\bu}_\varepsilon,\theta_\varepsilon]$ converge to a triple $[\widetilde{\varphi},\widetilde{\bu},\widetilde{\theta}]$ that solves the local version of system \eqref{system:nlCHOB} as $\varepsilon\to 0$. In particular, we shall show that $[\widetilde{\varphi},\widetilde{\bu},\widetilde{\theta}]$ satisfies --- in weak sense --- the Cahn--Hilliard--Bousinessq system similar to what was proposed in \cite{peralta2021}:
\begin{subequations} 
\begin{align}
  \partial_t\varphi + {\bu}\cdot\nabla\varphi = \Delta \mu,\quad  \mu = -\Delta\varphi + \eta F'(\varphi)+ \ell_c \theta, 
\end{align}   
\begin{align}
 \begin{aligned}
  \partial_t{\bu} + ({\bu}\cdot\nabla){\bu} {-}  \dive(\nu(\varphi)2\mathrm{D}{\bu}) + \nabla p = \mathcal{K}(\mu-\ell_c\theta)\nabla\varphi + \ell(\varphi,\theta){\bg} + {\bq},
  \end{aligned}
\end{align}
\begin{align}
 \partial_t\theta - \ell_h\partial_t\varphi + {\bu}\cdot\nabla(\theta - \ell_h\varphi) -\kappa \Delta\theta = {\bg}\cdot{\bu} + z.
\end{align}
\label{system:CHOB}
\end{subequations}
with the incompressibility condition $\dive{\bu} = 0$ in $Q$, the incorporated initial conditions $\varphi(0) = \varphi_0$, ${\bu}(0) = {\bu}_0$, and $\theta(0) = \theta_0$ in $\Omega$ and closed with the boundary conditions $\frac{\partial\mu}{\partial{\bn}} = 0$,  $ {\bu} = 0$, and $ \frac{\partial\theta}{\partial{\bn}}  = 0$ on $\Gamma$.
We mention that the total energy for the local system can be written as
\begin{align}
\begin{aligned}
& 2\widetilde{\mathbb{E}}(\varphi,{\bu},\theta):= \frac{1}{\ell_c}\int_\Omega \frac{1}{2}|\nabla\varphi(x)|^2+ F(\varphi(x))\du x   + \frac{1}{\mathcal{K}\ell_c}\int_\Omega |{\bu}(x)|^2 \du x + \frac{1}{\ell_h}\int_\Omega |\theta(x)|^2\du x,
\end{aligned}\label{totalenergy}
\end{align}
while the local energy is defined for any $\varphi\in V$ as
\begin{align*}
    \mathbb{E}_{l}(\varphi) := \frac{1}{2}\int_\Omega |\nabla\varphi(x)|^2\du x.
\end{align*}

Finally, the purpose of this section is to establish the following result.
\begin{theorem}
    Given $\varepsilon>0$ let $[\varphi_{\varepsilon,0},{\bu}_{\varepsilon,0},\theta_{\varepsilon,0}]\in H\times H_\sigma\times H$, and assume that there exists $m_\Omega\in(-1,1)$ such that $\widehat{\varphi_{\varepsilon,0}}=m_\Omega$ for any $\varepsilon>0$. Suppose that there exists $[\widetilde{\varphi}_{0},\widetilde{\bu}_{0},\widetilde{\theta}_{0}]\in V\times H_\sigma\times H$ such that the following convergences hold:
    \begin{align*}
        &\varphi_{\varepsilon,0}\to \widetilde\varphi_0 \text{ in }H, \quad{\bu}_{\varepsilon,0}\to \widetilde{\bu}_0 \text{ in }H_\sigma,\quad
        \theta_{\varepsilon,0}\to \widetilde\theta_0 \text{ in }H\\
        &\mathbb{E}^\varepsilon(\varphi_{\varepsilon,0},{\bu}_{\varepsilon,0},\theta_{\varepsilon,0}) \to \widetilde{\mathbb{E}}(\widetilde{\varphi}_{0},\widetilde{\bu}_{0},\widetilde{\theta}_{0}).
    \end{align*}
    If for any $\varepsilon>0$, $[\varphi_\varepsilon,{\bu}_\varepsilon,\theta_\varepsilon]$ solves \eqref{system:nlCHOB} in weak sense with initial conditions $[\varphi_{\varepsilon,0},{\bu}_{\varepsilon,0},\theta_{\varepsilon,0}]$, then as $\varepsilon\to0$
    \begin{align}
        &\varphi_\varepsilon \rightharpoonup \widetilde{\varphi}\text{ in }L^2(I;V)\label{asymp:phiV}\\
        &\varphi_\varepsilon \ws \widetilde{\varphi}\text{ in }L^\infty(I;H)\label{asymp:phiH}\\
        &\varphi_\varepsilon \to \widetilde{\varphi}\text{ in }L^2(I;H)\text{ and a.e. in }\Omega\times(0,T)\label{asymp:phistH}\\
        &{\bu}_\varepsilon \rightharpoonup \widetilde{\bu}\text{ in }L^2(I;V_\sigma)\label{asymp:uV}\\
        &{\bu}_\varepsilon \ws \widetilde{\bu}\text{ in }L^\infty(I;H_\sigma)\label{asymp:uH}\\
        &{\bu}_\varepsilon \to \widetilde{\bu}\text{ in }L^2(I;H_\sigma)\text{ and a.e. in }\Omega\times(0,T)\label{asymp:ustH}\\
        &\theta_\varepsilon \rightharpoonup \widetilde{\theta}\text{ in }L^2(I;V)\label{asymp:thetaV}\\
        &\theta_\varepsilon \ws \widetilde{\theta}\text{ in }L^\infty(I;H)\label{asymp:thetaH}\\
        &\theta_\varepsilon \to \widetilde{\theta}\text{ in }L^2(I;H)\text{ and a.e. in }\Omega\times(0,T)\label{asymp:thetastH}
    \end{align}
    where $[\widetilde{\varphi},\widetilde{\bu},\widetilde{\theta}]$ solves \eqref{system:CHOB} in weak sense with $[\widetilde{\varphi}(0),\widetilde{\bu}(0),\widetilde{\theta}(0)] = [\widetilde{\varphi}_{0},\widetilde{\bu}_{0},\widetilde{\theta}_{0}]$ and 
    \begin{align}
        &\widetilde{\varphi}\in W^{+\infty,4/d}(H;V^*)\cap L^2(I;V)\cap L^2(I;V_2) \\
        &\widetilde{\bu}\in W^{+\infty,4/d}(H_\sigma;V_\sigma^*)\cap L^2(I;V_\sigma)\\
        &\widetilde{\theta}\in W^{+\infty,4/d}(H;V^*)\cap L^2(I;V)
    \end{align}
    where $\widetilde{\mu} = -\Delta\widetilde{\varphi} + F'(\widetilde{\varphi}) + \ell_c\widetilde{\theta}\in L^2(I;H)$. Furthermore, for any $t\in[0,T]$ the following energy inequality holds
    \begin{align}
	\begin{aligned}
		&\widetilde{\mathbb{E}}(\widetilde{\varphi}(t),\widetilde{\bu}(t),\widetilde{\theta}(t)) + \int_0^t \mathbb{D}(\widetilde{\varphi}(s),\widetilde{\mu}(s),\widetilde{\bu}(s),\widetilde{\theta}(s)) \du s\\ &
        \le \widetilde{\mathbb{E}}(\widetilde{\varphi}_0,\widetilde{\bu}_0,\widetilde{\theta}_0)+ \int_0^t \big\{\langle{\bq}(s),\widetilde{\bu}(s)\rangle_{V_\sigma} + (\ell(\widetilde{\varphi}(s),\widetilde{\theta}(s)){\bg},\widetilde{\bu}(s))_\Omega\\ &\ \  + \langle{z}(s),\widetilde{\theta}(s)\rangle_{V} + ({\bg}\cdot\widetilde{\bu}(s),\widetilde{\theta}(s))_\Omega \big\}\du s  .
	\end{aligned}\label{energyineql}
	\end{align}
 \label{th:asymptotics}
\end{theorem}

By saying $[\widetilde{\varphi},\widetilde{\bu},\widetilde{\theta}]$ solves \eqref{system:CHOB} in weak sense, we mean it to satisfy the variational problem 
\begin{subequations}
\begin{align}
    \begin{aligned}
        \langle \partial_t\widetilde{\varphi}(t),\psi \rangle_V + &\, (\widetilde{\bu}(t)\cdot\nabla\widetilde{\varphi}(t),\psi)_\Omega + (\nabla\widetilde{\mu}(t),\nabla\psi)_\Omega = 0
    \end{aligned}\label{weak:locphi}
\end{align}
\begin{align}
    \begin{aligned}
        &\langle \partial_t\widetilde{\bu}(t),{\bv} \rangle_{V_\sigma} + ((\widetilde{\bu}(t)\cdot\nabla)\widetilde{\bu}(t),{\bv} )_\Omega + 2(\nu(\varphi)\mathrm{D}\widetilde{\bu}(t),\mathrm{D}{\bv})\\ & = \mathcal{K}(\bv\cdot\nabla\widetilde{\varphi}(t), (\widetilde{\mu}(t) - \ell_c\widetilde{\theta}(t)) )_\Omega + (\ell(\widetilde{\varphi}(t),\widetilde{\theta}(t)){\bg},{\bv} )_\Omega + \langle {\bq}(t),{\bv} \rangle_{V_\sigma}
    \end{aligned}\label{weak:locu}
\end{align}
\begin{align}
    \begin{aligned}
        & \langle \partial_t\widetilde{\theta}(t),\vartheta  \rangle_V -  \ell_h\langle \partial_t\widetilde{\varphi}(t),\vartheta \rangle_V  + \kappa(\nabla\widetilde{\theta}(t),\nabla\vartheta)\\ & = (\widetilde{\bu}(t)\cdot\nabla(\ell_h\widetilde{\varphi}(t)-\widetilde{\theta}(t)),\vartheta)_\Omega + ({\bg}\cdot\widetilde{\bu}(t),\vartheta)_\Omega + \langle z(t),\vartheta\rangle_V
    \end{aligned} \label{weak:loctheta}
\end{align}
\end{subequations}
for all $\psi\in V$, ${\bv}\in V_\sigma$, $\vartheta\in V$ and almost every $t\in (0,T)$.

 To proceed with the  proof of the theorem above we shall need the following Lemma which was proven in \cite[Lemma 3.3]{DAVOLI2021a}.
\begin{lemma}
    If $\varphi_1,\varphi_2\in V$ then 
        \begin{align}
            &\lim_{\varepsilon\to 0} \mathbb{E}^\varepsilon_{nl}(\varphi_1) = \mathbb{E}_{l}(\varphi_1),\label{conv:nontoloc}\\
            &\lim_{\varepsilon\to 0}\int_\Omega (a_\varepsilon(x)\varphi_1(x) - J_\varepsilon\ast\varphi_1 (x) )\varphi_2(x)\du x = \int_\Omega \nabla\varphi_1(x)\cdot\nabla\varphi_2(x)\du x.
        \end{align}
    Furthermore, if $\{\varphi_\varepsilon\}\subset H$ is a sequence that converges strongly to $\varphi\in H$ in $H$ then 
    \begin{align}
        \mathbb{E}_{l}(\varphi) \le \liminf_{\varepsilon\to 0} \mathbb{E}^\varepsilon_{nl}(\varphi_\varepsilon).\label{gamma:energy}
    \end{align}
    \label{lemma:davoli}
\end{lemma}

We mention that due to the computations in the previous section, the proof of the main theorem will be simplified. Mainly because the arguments for establishing the uniform boundedness of the discretized solutions can be done analogously for the solutions of the nonlocal system with the newly specified convolution kernels.
We are now in the position to prove the main result of this section. 
\begin{proof}[Proof of Theorem \ref{th:asymptotics}]
Since every solution $[\varphi_\varepsilon,{\bu}_\varepsilon,\theta_\varepsilon]$ satisfies the energy inequality \eqref{energyineq}, following the same arguments to be able to reach \eqref{estimate:completeN} we get
\begin{align*}
    \begin{aligned}
        &\sup_{t\in [0,T]}\widehat{\mathbb{E}}(\varphi_\varepsilon(t),{\bu}_\varepsilon(t),\theta_\varepsilon(t))+ \int_0^{T} \widehat{\mathbb{D}}(\mu_\varepsilon(s),{\bu}_\varepsilon(s),\theta_\varepsilon(s))\du s \le M,
    \end{aligned}
\end{align*}
for some constant $M>0$ independent of $\varepsilon>0$. This provides uniform boundedness of $\{\varphi_\varepsilon \}$ in $L^\infty(I;H)$ and $L^\infty(I;H)$, $\{{\bu}_\varepsilon\}$ in $L^\infty(I;H_\sigma)\cap L^2(I;V_\sigma)$ and $\{\nabla\mu_\varepsilon\}$ in $L^2(I;L^2(\Omega)^d)$. Testing \eqref{weak:theta} with $\vartheta = 1$ and integrating over $(0,T)$ we get
\begin{align*}
    \left| \int_\Omega \theta_\varepsilon \du x \right| \le c\left( \|{\bu_\varepsilon}\|_{L^\infty(I;H_\sigma)} + \|z\|_{L^2(I;V^*)} \right),
\end{align*}
and thus the uniform boundedness of $\{\theta_\varepsilon \}$ in $L^2(I;V)$.

Following the same steps as in the proof of Theorem \ref{th:existence}, we can show that \eqref{energy:incomplete} holds but with the superscript $n$ replaced with the subscript $\varepsilon$ and the superscripts on the external force ${\bq}^n$ and external heat source $z^n$ removed. Taking the integral over $(0,T)$ on such estimate and using the assumption on convergence of the nonlocal total energy to the local total energy evaluated at the initial data shows that $\{F(\varphi_\varepsilon)\}$ is uniformly bounded in $L^\infty(I;L^1(\Omega))$. From assumption (A5), we thus get
\begin{align*}
    \begin{aligned}
        \left|\int_\Omega \mu_\varepsilon \du x \right| & = \left|\int_\Omega a_\varepsilon\varphi_\varepsilon - J\ast\varphi_\varepsilon + F'(\varphi_\varepsilon) + \ell_c\theta_\varepsilon \du x \right|\\
       & \le c_3 \int_\Omega |F(\varphi_\varepsilon)| \du x + c_4|\Omega| + \ell_c\sqrt{|\Omega|}\|\theta_\varepsilon\|.
    \end{aligned}
\end{align*}
    which implies uniform boundedness of $\{ \mu_\varepsilon\}$ in $L^2(I;V)$.

    From these, we establish the existence of $[\widetilde{\varphi},\widetilde{\bu},\widetilde{\theta}]$ and $\widetilde{\mu}$ such that \eqref{asymp:phiV}, \eqref{asymp:phiH}, \eqref{asymp:uV}, \eqref{asymp:uH}, \eqref{asymp:thetaV}, \eqref{asymp:thetaH} and $\mu_\varepsilon\rightharpoonup\widetilde{\mu}$ in $L^2(I;V)$. Following analogous arguments as in the well-posedness of the nonlocal system, we derive uniform boundedness for the time derivatives of the variables $[\varphi_\varepsilon,{\bu}_\varepsilon,\theta_\varepsilon]$, i.e.
    \begin{align}
        \|\partial_t\varphi_\varepsilon\|_{L^{4/d}(I;V^*)} + \|\partial_t{\bu}_\varepsilon\|_{L^{4/d}(I;V_\sigma^*)} + \|\partial_t\theta_\varepsilon\|_{L^{4/d}(I;V^*)} \le M.
    \end{align}
   Due to Aubin-Lions-Simon we infer that $W^{+\infty,4/d}(H_\sigma,V^*_\sigma)\hookrightarrow C(\overline{I};H_\sigma)$, $W^{+\infty,4/d}(H,V^*)\hookrightarrow C(\overline{I};H)$, $W^{2,4/d}(V_\sigma,V^*_\sigma)\hookrightarrow L^2({I};H_\sigma)$ and $W^{2,4/d}(V,V^*)\hookrightarrow L^2({I};H)$ are compact which give us the strong and a.e. convergences \eqref{asymp:phistH}, \eqref{asymp:ustH} and \eqref{asymp:thetastH}. The convergences we have in hand, by passing through the limits, are enough to establish that $[\widetilde{\varphi},\widetilde{\bu},\widetilde{\theta}]$ satisfies \eqref{weak:locphi}, \eqref{weak:locu} and \eqref{weak:loctheta}.
 
    Multiplying $G'(\varphi_\varepsilon)$ and $\mu_\varepsilon$, using the second equation in \eqref{orderundpotential}, integrating over $\Omega$, and rearranging the resulting integral, one gets by virtue of H{\"o}lder and Young inequalities
    \begin{align*}
    \begin{aligned}
        &\int_\Omega |G'(\varphi_\varepsilon(x))|^2\du x + \frac{1}{2}\int_\Omega\int_\Omega J_\varepsilon(x-y)\delta G'_\varphi(x,y)(\varphi_\varepsilon(x) - \varphi_\varepsilon(y)) \du x\du y\\
        &=  \int_\Omega ( \mu_\varepsilon(x) + a^*\varphi_\varepsilon(x) - \ell_c\theta_\varepsilon(x) ) G'(\varphi_\varepsilon(x))\du x \le M + \frac{1}{2}\int_\Omega |G'(\varphi_\varepsilon(x))|^2\du x,
    \end{aligned}
    \end{align*}
    where $M>0$ is a constant independent of $\varepsilon>0$ and $\delta G'_\varphi(x,y) := (G'(\varphi_\varepsilon(x)) - G'(\varphi_\varepsilon(y)))$. Since $G$ is a strictly convex function, we see the the second term on the left-hand side above is nonnegative, and thus the uniform boundedness of   $\{F'(\varphi_\varepsilon)\}$ in  $ L^2(I;H)$.
    This implies the existence of ${\blf}\in L^2(I;H)$ such the $F'(\varphi_\varepsilon)\rightharpoonup {\blf}$ in $L^2(I;H)$. Furthermore, by virtue of Egorov's theorem (see footnote in \cite[p. 1093]{abels2015}), the strong convergence \eqref{asymp:phistH} implies $F'(\varphi_\varepsilon)\to F'(\widetilde{\varphi})$ a.e., and hence in $L^2(I;L^2(\Omega))$. On the other hand, since $\mu_\varepsilon$ and $\theta_\varepsilon$ are bounded in $L^2(I;H)$, with independence on $\varepsilon>0$, we get uniform boundedness of $a_\varepsilon(x)\varphi_\varepsilon - J_\varepsilon\ast\varphi_\varepsilon$ in $L^2(I;H)$ which implies the existence of ${\bg}\in L^2(I;H)$ such that $a_\varepsilon(x)\varphi_\varepsilon - J_\varepsilon\ast\varphi_\varepsilon \rightharpoonup {\bg}$ in $L^2(I;H)$.

    Referring to \cite[Lemma 3.2]{DAVOLI2021a} $\mathbb{E}^\varepsilon_{nl}$ is G{\^a}teaux differentiable which can be computed for any $\varphi,\psi$ as
    \begin{align*}
    \begin{aligned}
        \langle D\mathbb{E}_{nl}^\varepsilon(\varphi),\psi \rangle & = \frac{1}{2}\int_\Omega\int_\Omega J_\varepsilon(x-y)(\varphi(x)-\varphi(y))(\psi(x)-\psi(y)) \du x\du y\\
        & = \int_\Omega (a_\varepsilon(x) \varphi(x) - J_\varepsilon\ast\varphi(x))\psi(x) \du x.
    \end{aligned}
    \end{align*}
    The second order G{\^a}teaux derivative as a bilinear form in $H$ can also be computed for any $\varphi,\psi_1,\psi_2$ as
    \begin{align*}
        \langle D^2\mathbb{E}_{nl}^\varepsilon(\varphi),[\psi_1,\psi_2] \rangle = \frac{1}{2}\int_\Omega\int_\Omega J_\varepsilon(x-y)(\psi_1(x)-\psi_1(y))(\psi_2(x)-\psi_2(y)) \du x\du y.
    \end{align*}
    The nonegativity of $J_\varepsilon$ thus imply that $\langle D^2\mathbb{E}_{nl}^\varepsilon(\varphi),[\psi,\psi] \rangle \ge 0$ for any $\varphi,\psi\in H$. Letting $\psi\in L^2(I;V)$ such that $\widehat{\psi} = m_\Omega$, we take the Taylor's expansion of $\mathbb{E}_{nl}^\varepsilon$ about $\varphi_\varepsilon$ so that
    \begin{align}
        \begin{aligned}
            & \int_0^T\mathbb{E}_{nl}^\varepsilon(\psi(t)) \du t  = \int_0^T \mathbb{E}_{nl}^\varepsilon(\varphi_\varepsilon(t)) + \langle D\mathbb{E}_{nl}^\varepsilon(\varphi_\varepsilon(t)),(\psi(t) - \varphi_\varepsilon(t)) \rangle \\ &\ \ 
            + \frac{1}{2}\langle D^2\mathbb{E}_{nl}^\varepsilon(\widetilde{\varphi}),(\psi(t) - \varphi_\varepsilon(t))^2 \rangle \du t\\
            & \ge \int_0^T\mathbb{E}_{nl}^\varepsilon(\varphi_\varepsilon(t)) +  (a_\varepsilon(x) \varphi_\varepsilon(t) - J_\varepsilon\ast\varphi_\varepsilon(t),\psi(t) - \varphi_\varepsilon(t))_\Omega \du t.
        \end{aligned}
    \end{align}
    Taking advantage of the convergences \eqref{conv:nontoloc} and \eqref{gamma:energy} in Lemma \ref{lemma:davoli}, Fatou's lemma, and the convergence $\varphi_\varepsilon\to \widetilde\varphi$ in $C(\overline{I};H)$ we infer that
    \begin{align}
        \begin{aligned}
              \int_0^T\mathbb{E}_{l}(\widetilde{\varphi}(t)) +  ({\bg}(t),\psi(t) - \widetilde{\varphi}(t))_\Omega \du t \le \int_0^T\mathbb{E}_{l}(\psi(t)) \du t
        \end{aligned}
    \end{align}
    This implies that ${\bg}\in\partial\mathbb{E}_{l}(\widetilde{\varphi})$, from which we further infer --- by additionally considering the differentiability and convexity of $\mathbb{E}_l$ --- that $\langle {\bg},\psi \rangle = \langle D\mathbb{E}_l(\widetilde{\varphi}), \psi \rangle $ for any $\psi\in L^2(I;V)$, i.e.
    \begin{align}
        \int_0^T \int_\Omega {\bg} \psi \du x\du t = \int_0^T \int_\Omega \nabla\widetilde{\varphi}\cdot\nabla\psi \du x\du t,
    \end{align}
    and we further conclude that ${\bg} = -\Delta \widetilde{\varphi}$ in $L^2(I;H)$ and $\frac{\partial\widetilde{\varphi}}{\partial {\bn}} = 0$ almost everywhere in $\Gamma$, for more details see \cite[pages 142--143]{Davoli2021b} for example. 

    Summarizing the computations above, we were able to show that $\widetilde{\mu} = -\Delta \widetilde{\varphi} + F'(\widetilde{\varphi}) + \ell_c\widetilde{\theta}$ in $L^2(I;H)$ and that $\widetilde\varphi\in L^2(I;V_2)$.

    Finally, passing to the limits would show that the equations \eqref{weak:locphi}--\eqref{weak:loctheta} hold. While, Lemma \ref{lemma:davoli}, the weak lower semicontinuity of norms, the assumption on the energy convergence of the initial data, and the convergences we have in hand show that \eqref{energyineql} holds.
    
\end{proof}

\section*{Acknowledgement}
\v S. N. and J.S.H.S have been supported by  Præmium Academiæ of \v S. Ne\v casov\' a. Moreover, \v S. N. has been supported by by the Czech Science Foundation (GA\v CR) through project GA22-01591S. The Institute of Mathematics CAS is supported by RVO:67985840.


\begin{thebibliography}{10}
	\expandafter\ifx\csname url\endcsname\relax
	\def\url#1{\texttt{#1}}\fi
	\expandafter\ifx\csname urlprefix\endcsname\relax\def\urlprefix{URL }\fi
	\expandafter\ifx\csname href\endcsname\relax
	\def\href#1#2{#2} \def\path#1{#1}\fi
	
	\bibitem{cahn1958}
	J.~W. Cahn, J.~E. Hilliard, Free energy of a nonuniform system. {I}.
	interfacial free energy, The Journal of Chemical Physics 28~(2) (1958)
	258--267.
	\newblock \href {https://doi.org/10.1063/1.1744102}
	{\path{doi:10.1063/1.1744102}}.
	
	\bibitem{cahn1959}
	J.~Cahn, J.~Hilliard, {Free Energy of a Nonuniform System. III. Nucleation in a
		Two‐Component Incompressible Fluid}, The Journal of Chemical Physics 31~(3)
	(2004) 688--699.
	\newblock \href {https://doi.org/10.1063/1.1730447}
	{\path{doi:10.1063/1.1730447}}.
	
	\bibitem{rocca2023}
	E.~Rocca, G.~Schimperna, A.~Signori, On a Cahn–Hilliard–Keller–Segel
	model with generalized logistic source describing tumor growth, Journal of
	Differential Equations 343 (2023) 530--578.
	\newblock \href {https://doi.org/https://doi.org/10.1016/j.jde.2022.10.026}
	{\path{doi:https://doi.org/10.1016/j.jde.2022.10.026}}.
	
	\bibitem{rocca2021}
	E.~Rocca, L.~Scarpa, A.~Signori, Parameter identification for nonlocal phase
	field models for tumor growth via optimal control and asymptotic analysis,
	Mathematical Models and Methods in Applied Sciences 31~(13) (2021)
	2643--2694.
	\newblock \href {https://doi.org/10.1142/S0218202521500585}
	{\path{doi:10.1142/S0218202521500585}}.
	
	\bibitem{abels2009a}
	H.~Abels, On a diffuse interface model for two-phase flows of viscous,
	incompressible fluids with matched densities, Archive for Rational Mechanics
	and Analysis 194~(2) (2009) 463--506.
	\newblock \href {https://doi.org/10.1007/s00205-008-0160-2}
	{\path{doi:10.1007/s00205-008-0160-2}}.
	
	\bibitem{abels2022}
	H.~Abels, Y.~Terasawa, Convergence of a nonlocal to a local diffuse interface
	model for two-phase flow with unmatched densities (2022).
	\newblock \href {https://doi.org/10.3934/dcdss.2022117}
	{\path{doi:10.3934/dcdss.2022117}}.
	
	\bibitem{kalousek2021}
	M.~Kalousek, S.~Mitra, A.~Schlömerkemper, Existence of weak solutions of
	diffuse interface models for magnetic fluids, PAMM 21~(1) (2021) e202100205.
	\newblock \href {https://doi.org/https://doi.org/10.1002/pamm.202100205}
	{\path{doi:https://doi.org/10.1002/pamm.202100205}}.
	
	\bibitem{mitra2023}
	M.~Kalousek, S.~Mitra, A.~Schl{\"o}merkemper, Existence of weak solutions to a
	diffuse interface model involving magnetic fluids with unmatched densities,
	Nonlinear Differential Equations and Applications NoDEA 30~(4) (2023) 52.
	\newblock \href {https://doi.org/10.1007/s00030-023-00852-0}
	{\path{doi:10.1007/s00030-023-00852-0}}.
	
	\bibitem{ohta1986}
	T.~Ohta, K.~Kawasaki, Equilibrium morphology of block copolymer melts,
	Macromolecules 19~(10) (1986) 2621--2632.
	\newblock \href {https://doi.org/10.1021/ma00164a028}
	{\path{doi:10.1021/ma00164a028}}.
	
	\bibitem{bertozzi2007}
	A.~L. Bertozzi, S.~Esedoglu, A.~Gillette, Inpainting of binary images using the
	cahn–hilliard equation, IEEE Transactions on Image Processing 16~(1) (2007)
	285--291.
	\newblock \href {https://doi.org/10.1109/TIP.2006.887728}
	{\path{doi:10.1109/TIP.2006.887728}}.
	
	\bibitem{Giacomin1997}
	G.~Giacomin, J.~Lebowitz, Phase segregation dynamics in particle systems with
	long range interactions. i. macroscopic limits, Journal of Statistical
	Physics 87~(1) (1997) 37--61.
	\newblock \href {https://doi.org/10.1007/BF02181479}
	{\path{doi:10.1007/BF02181479}}.
	
	\bibitem{Giacomin1998}
	G.~Giacomin, J.~Lebowitz, Phase segregation dynamics in particle systems with
	long range interactions ii: Interface motion, SIAM Journal on Applied
	Mathematics 58~(6) (1998) 1707--1729.
	\newblock \href {https://doi.org/10.1137/S0036139996313046}
	{\path{doi:10.1137/S0036139996313046}}.
	
	\bibitem{abels2015}
	H.~Abels, S.~Bosia, M.~Grasselli, Cahn--{H}illiard equation with nonlocal
	singular free energies, Annali di Matematica Pura ed Applicata (1923 -)
	194~(4) (2015) 1071--1106.
	
	\bibitem{bates2005}
	P.~W. Bates, J.~Han, The {N}eumann boundary problem for a nonlocal
	{C}ahn–{H}illiard equation, Journal of Differential Equations 212~(2)
	(2005) 235--277.
	\newblock \href {https://doi.org/https://doi.org/10.1016/j.jde.2004.07.003}
	{\path{doi:https://doi.org/10.1016/j.jde.2004.07.003}}.
	
	\bibitem{burkovska2021}
	O.~Burkovska, M.~Gunzburger, On a nonlocal {C}ahn–{H}illiard model permitting
	sharp interfaces, Mathematical Models and Methods in Applied Sciences 31~(09)
	(2021) 1749--1786.
	\newblock \href {https://doi.org/10.1142/S021820252150038X}
	{\path{doi:10.1142/S021820252150038X}}.
	
	\bibitem{colli2012}
	P.~Colli, S.~Frigeri, M.~Grasselli, Global existence of weak solutions to a
	nonlocal {C}ahn–{H}illiard–{N}avier–{S}tokes system, Journal of
	Mathematical Analysis and Applications 386~(1) (2012) 428--444.
	\newblock \href {https://doi.org/https://doi.org/10.1016/j.jmaa.2011.08.008}
	{\path{doi:https://doi.org/10.1016/j.jmaa.2011.08.008}}.
	
	\bibitem{DAVOLI2021a}
	E.~Davoli, L.~Scarpa, L.~Trussardi, Local asymptotics for nonlocal convective
	{C}ahn-{H}illiard equations with ${W}^{1,1}$ kernel and singular potential,
	Journal of Differential Equations 289 (2021) 35--58.
	\newblock \href {https://doi.org/https://doi.org/10.1016/j.jde.2021.04.016}
	{\path{doi:https://doi.org/10.1016/j.jde.2021.04.016}}.
	
	\bibitem{Davoli2021b}
	E.~Davoli, L.~Scarpa, L.~Trussardi, Nonlocal-to-local convergence of
	{C}ahn--{H}illiard equations: {N}eumann boundary conditions and viscosity
	terms, Archive for Rational Mechanics and Analysis 239~(1) (2021) 117--149.
	\newblock \href {https://doi.org/10.1007/s00205-020-01573-9}
	{\path{doi:10.1007/s00205-020-01573-9}}.
	
	\bibitem{frigeri2016}
	S.~Frigeri, C.~Gal, M.~Grasselli, On nonlocal
	{C}ahn--{H}illiard--{N}avier--{S}tokes systems in two dimensions, Journal of
	Nonlinear Science 26~(4) (2016) 847--893.
	
	\bibitem{gurtin1996}
	M.~Gurtin, D.~Polignone, J.~Vi\~{n}als, Two-phase binary fluids and immiscible
	fluids described by an order parameter, Mathematical Models and Methods in
	Applied Sciences 06~(06) (1996) 815--831.
	\newblock \href {https://doi.org/10.1142/S0218202596000341}
	{\path{doi:10.1142/S0218202596000341}}.
	
	\bibitem{jasnow1996}
	D.~Jasnow, J.~Viñals, {Coarse‐grained description of thermo‐capillary
		flow}, Physics of Fluids 8~(3) (1996) 660--669.
	\newblock \href {https://doi.org/10.1063/1.868851}
	{\path{doi:10.1063/1.868851}}.
	
	\bibitem{abels2009b}
	H.~Abels, Existence of weak solutions for a diffuse interface model for
	viscous, incompressible fluids with general densities, Communications in
	Mathematical Physics 289~(1) (2009) 45--73.
	\newblock \href {https://doi.org/10.1007/s00220-009-0806-4}
	{\path{doi:10.1007/s00220-009-0806-4}}.
	
	\bibitem{garcke2012}
	H.~Abels, H.~Garcke, G.~Gr{\"u}n, Thermodynamically consistent, frame
	indifferent diffuse interface models for incompressible two-phase flows with
	different densities, Mathematical Models and Methods in Applied Sciences
	22~(03) (2012) 1150013.
	\newblock \href {https://doi.org/10.1142/S0218202511500138}
	{\path{doi:10.1142/S0218202511500138}}.
	
	\bibitem{boyer2002}
	F.~Boyer, Theoretical and numerical study of multi-phase flows through order
	parameter formulation, pp. 488--490.
	\newblock \href {https://doi.org/10.1142/9789812792617_0091}
	{\path{doi:10.1142/9789812792617_0091}}.
	
	\bibitem{boyer2001}
	F.~Boyer, Nonhomogeneous {C}ahn–{H}illiard fluids, Ann. Inst. H. Poincar{\'e}
	Anal. Non Lin{\'e}aire 18~(2) (2001) 225--259.
	\newblock \href {https://doi.org/10.1016/S0294-1449(00)00063-9}
	{\path{doi:10.1016/S0294-1449(00)00063-9}}.
	
	\bibitem{gal2010}
	C.~Gal, M.~Grasselli, Asymptotic behavior of a
	{C}ahn–{H}illiard–{N}avier–{S}tokes system in 2d, Ann. Inst. H.
	Poincar{\'e} Anal. Non Lin{\'e}aire 27~(1) (2010) 401--436.
	\newblock \href {https://doi.org/10.1016/J.ANIHPC.2009.11.013}
	{\path{doi:10.1016/J.ANIHPC.2009.11.013}}.
	
	\bibitem{giorgini2019}
	A.~Giorgini, A.~Miranville, R.~Temam, Uniqueness and regularity for the
	{N}avier--{S}tokes--{C}ahn--{H}illiard system, SIAM Journal on Mathematical
	Analysis 51~(3) (2019) 2535--2574.
	\newblock \href {https://doi.org/10.1137/18M1223459}
	{\path{doi:10.1137/18M1223459}}.
	
	\bibitem{caginalp1986}
	G.~Caginalp, An analysis of a phase field model of a free boundary, Archive for
	Rational Mechanics and Analysis 92~(3) (1986) 205--245.
	\newblock \href {https://doi.org/10.1007/BF00254827}
	{\path{doi:10.1007/BF00254827}}.
	
	\bibitem{peralta2021}
	G.~Peralta, Distributed optimal control of the 2d
	{C}ahn--{H}illiard--{O}berbeck--{B}oussinesq system for nonisothermal viscous
	two-phase flows, Applied Mathematics {\&} Optimization 84~(2) (2021)
	1219--1279.
	\newblock \href {https://doi.org/10.1007/s00245-021-09759-7}
	{\path{doi:10.1007/s00245-021-09759-7}}.
	
	\bibitem{Peralta2022}
	G.~Peralta, Weak and very weak solutions to the viscous
	{C}ahn--{H}illiard--{O}berbeck--{B}oussinesq phase-field system on
	two-dimensional bounded domains, Journal of Evolution Equations 22~(1) (2022)
	12.
	\newblock \href {https://doi.org/10.1007/s00028-022-00765-y}
	{\path{doi:10.1007/s00028-022-00765-y}}.
	
	\bibitem{melchionna2019}
	S.~Melchionna, H.~Ranetbauer, L.~Scarpa, L.~Trussardi, From nonlocal to local
	{C}ahn-{H}illiard equation, Adv. Math. Sci. Appl. 28~(2) (2019) 197--211.
	
	\bibitem{bourgain2001}
	J.~Bourgain, H.~Brezis, P.~Mironescu, Another look at sobolev spaces, in: A.~S.
	J.L.~Menaldi, . E.~Rofman (Ed.), Optimal Control and Partial Differential
	Equations, IOS, Amsterdam, 2001, pp. 439--455.
	
	\bibitem{bourgain2002}
	J.~Bourgain, H.~Brezis, P.~Mironescu, Limiting embedding theorems for
	${W}^{s,p}$ when ${s}\uparrow {1}$ and applications, Journal d'Analyse
	Math{\'e}matique 87~(1) (2002) 77--101.
	
	\bibitem{ponce2003}
	A.~Ponce, A variant of {P}oincaré's inequality, Comptes Rendus Mathematique
	337~(4) (2003) 253--257.
	\newblock \href {https://doi.org/https://doi.org/10.1016/S1631-073X(03)00313-3}
	{\path{doi:https://doi.org/10.1016/S1631-073X(03)00313-3}}.
	
	\bibitem{ponce2004}
	A.~Ponce, A new approach to {S}obolev spaces and connections to
	$\gamma$-convergence, Calculus of Variations and Partial Differential
	Equations 19~(3) (2004) 229--255.
	
	\bibitem{temam1977}
	R.~Temam, Navier-Stokes Equations: Theory and Numerical Analysis, 3rd Edition,
	AMS Chelsea Publishing, Providence, Rhode Island, 2001.
	
\end{thebibliography}

\end{document}